\documentclass[final,reqno,twoside,onecolumn]{siamltex}

\usepackage{amssymb}
\usepackage{amsmath,xcolor}
\usepackage{mathrsfs}
\usepackage{graphicx}% needed for \resizebox
\usepackage{calc}%
\usepackage{multirow}
\usepackage[T1]{fontenc} % Use 8-bit encoding that has 256 glyphs
\usepackage{fourier} 

\newcommand{\Aa}{\mathcal A}
\newcommand{\Bb}{\mathcal B}
\newcommand{\Ff}{\mathcal F}
\newcommand{\Gg}{\mathcal G}
\newcommand{\Hh}{\mathcal H}
\newcommand{\Dt}{\Delta t}

\newcommand{\varep}{\varepsilon}

\newcommand{\norm} [1]{\left\| {#1}\right\|}

\newcommand{\R} {\mathbb R}
\newcommand{\Rplus} {\mathbb R_{+}}
\newcommand{\rhot} {\tilde {\rho}}
\newcommand{\ft} {\tilde{f}}

\newcommand{\intd}[1]{\left( #1 \right)}

\newcommand{\ddt}{\frac d{dt} }
\newcommand{\eqdef}{\overset{\mathrm{def}}{=\joinrel=}}

\def\beq{\begin{equation}}                                                                                                                                                                                                                                                                                                                                                                                                                                                                                                                                                                         
\def\eeq{\end{equation}} 

\def\beqs{\begin{equation*}}
\def\eeqs{\end{equation*}}

\def\bals{\begin{align*}}
\def\eals{\end{align*}}

\def\bspl{\begin{split}}
\def\espl{\end{split}}

\def\myclearpage{}

\title{Numerical analysis for generalized Forchheimer flows of slightly compressible fluids in porous media}
\author{Thinh Kieu \footnotemark[2]  }

\begin{document}

\maketitle
 
\renewcommand{\thefootnote}{\fnsymbol{footnote}}
\footnotetext[2]{Department of Mathematics, University of North Georgia, Gainesville Campus, 3820 Mundy Mill Rd., Oakwood, GA 30566, U.S.A. ({\tt thinh.kieu@ung.edu}).}               
               
\begin{abstract} In this paper, we will consider the generalized Forchheimer flows for slightly compressible fluids. Using Muskat's and Ward's general form of Forchheimer equations, we describe the fluid dynamics by a nonlinear degenerate parabolic equation for  density. 
The long time numerical approximation of the nonlinear degenerate parabolic equation with time dependent boundary conditions is studied.  The stability for all positive time is established in both a continuous time scheme and a discrete backward Euler scheme. A Gronwall's inequality-type is used to study the asymptotic behavior of the solution. 
 %The prior estimates for the solution, its time derivative in $L^2(\R_+;L^2(\Omega))$ and for gradient of solution in $L^\infty(0,T;W^{1,2-a}(\Omega)),$ with $a\in (0,1)$ are established. 
 Error estimates for the solution are derived for  both continuous and discrete time procedures. Numerical experiments confirm the theoretical analysis regarding convergence rates.  
 \end{abstract}
            
 %\category{...}{...}{...}
            
%\terms{...} 
            


\pagestyle{myheadings}
\thispagestyle{plain}
\markboth{Thinh Kieu}{Numerical analysis for non-Darcy flows in porous media}
            
%\begin{bottomstuff} 
%...
%\end{bottomstuff}
            
\myclearpage    
\section {Introduction }
The most common equation to describe fluid flows in porous media is the Darcy law   
\beq\label{Darcy}
-\nabla p = \frac {\mu}{\kappa} v,
\eeq
where $p$, $v$, $\mu$, $\kappa$ are, respectively (resp.), the pressure, velocity, absolute viscosity and permeability. 

When the Reynolds number is large, Darcy's law becomes invalid, see \cite{Muskatbook,BearBook}. A nonlinear relationship between the velocity and gradient of pressure is introduced by adding the higher order terms of velocity to Darcy's law. 
%It is known as generalized Forchheimer laws or non-Darcy's law, see \cite{Muskatbook,Ward64,BearBook,NieldBook,StraughanBook} and references therein. 
Forchheimer established this in \cite{ForchheimerBook} the following three nonlinear empirical models:  
\beq\label{2term}
-\nabla p=av+b|v|v,
\quad
-\nabla p=av+b|v|v+c |v|^2 v,
\quad
-\nabla p=av+d|v|^{m-1}v, m\in(1, 2).
\eeq
Above, the positive constants $a,b,c,d$ are obtained from experiments.

The generalized Forchheimer equation of \eqref{Darcy} and \eqref{2term} were proposed in \cite{ABHI1,HI1,HI2} of the form 
\beq\label{gF}
-\nabla p =\sum_{i=0}^N a_i |v|^{\alpha_i}v. 
\eeq  

These equations are analyzed numerically in \cite{Doug1993,EJP05,K1},
theoretically in \cite{ABHI1,HI2,HIKS1,HKP1,HK1,HK2} for single phase flows, and also in \cite{HIK1,HIK2} for two-phase flows.

In order to take into account the presence of density in the generalized Forchheimer equation, we modify \eqref{gF} using the dimensional analysis by Muskat \cite{Muskatbook} and Ward \cite{Ward64}. They proposed the following equation for both laminar and turbulent flows in porous media:
\beq\label{W}
-\nabla p =F(v^\alpha \kappa^{\frac {\alpha-3} 2} \rho^{\alpha-1} \mu^{2-\alpha}),\text{ where  $F$ is a function of one variable.}
\eeq 
In particular, when $\alpha=1,2$, Ward \cite{Ward64} established from experimental data that
\beq\label{FW} 
-\nabla p=\frac{\mu}{\kappa} v+c_F\frac{\rho}{\sqrt \kappa}|v|v,\quad \text{where }c_F>0.
\eeq

Combining  \eqref{gF} with the suggestive form \eqref{W} for the dependence on $\rho$ and $v$, we propose the following equation 
 \beq\label{FM}
-\nabla p= \sum_{i=0}^N a_i \rho^{\alpha_i} |v|^{\alpha_i} v,
 \eeq
where $N\ge 1$, $\alpha_0=0<\alpha_1<\ldots<\alpha_N$ are real numbers, the coefficients $a_0, \ldots, a_N$ are positive.  
Here, the viscosity and permeability are considered constant, and we do not specify the dependence of $a_i$'s on them.
%Our mathematical exposition below will allow all $\alpha_i\ge 0$  in \eqref{FM}. 

Multiplying both sides of the previous equation to $\rho$, we obtain 
 \beq\label{eq1}
 g( |\rho v|) \rho v   =-\rho\nabla p,
 \eeq
where the function $g$ is a generalized polynomial with non-negative coefficients. More precisely,  the function $g:\mathbb{R}^+\rightarrow\mathbb{R}^+$ is of the form
\beq\label{eq2}
g(s)=a_0s^{\alpha_0} + a_1s^{\alpha_1}+\cdots +a_Ns^{\alpha_N},\quad s\ge 0, 
\eeq 
where $N\ge 1,\alpha_0=0<\alpha_1<\ldots<\alpha_N$ are fixed real numbers, the coefficients $a_0, \ldots, a_N$ are non-negative numbers with $a_0>0$ and $a_N>0$.  

For slightly compressible fluids, the state equation is
\beq
\frac{d\rho}{dp}=\frac {\rho}{\kappa},
\eeq
which yields 
\beq \label{eq3}
\rho \nabla p=\kappa\nabla \rho.
\eeq

It follows form \eqref{eq2} and \eqref{eq3} that  
 \beq\label{Gs}
 g( |\rho v|) \rho v   =-\kappa\nabla \rho .
 \eeq
%By rescaling coefficients of $g$ we have
%\[
% g( |\rho v|) \rho v   =-\nabla \rho.
%\]

Solving  for $\rho v$ from \eqref{Gs} gives 
\beq\label{ru} 
\rho v=- \kappa K(|\kappa\nabla \rho |)\nabla \rho,
\eeq
where the function $K: \mathbb{R}^+\rightarrow\mathbb{R}^+$ is defined for $\xi\ge 0$ by
\beq\label{Kdef}
K(\xi)=\frac{1}{g(s(\xi))}, \text{ with }  s=s(\xi) \text{  being the unique non-negative solution of } sg(s)=\xi.
\eeq

The continuity equation is
\beq\label{con-law}
\phi\rho_t+{\rm div }(\rho v)=f,
\eeq
where the constant  $\phi\in(0,1)$ is the porosity, $f$ is the external mass flow rate . 

Combining \eqref{ru} and \eqref{con-law}, we obtain 
\beq\label{utporo}
\phi\rho_t-\kappa\nabla \cdot(K(|\kappa\nabla \rho|)\nabla \rho) = f.
\eeq
Then by scaling the time variable in \eqref{utporo}, we can assume that the multiple factor is $1$.  Hence \eqref{utporo} becomes 
\beq\label{maineq}
\rho_t- \nabla\cdot(K(|\nabla \rho|)\nabla \rho)=f.
\eeq
This equation is a nonlinear degenerate parabolic equation as the density gradient goes to infinity. For the existence and regularity theory of degenerate parabolic equations, see e.g. \cite{MR2566733, LadyParaBook68,HIKS1}. 

The numerical analysis of the degenerate parabolic equation arising in flow in porous media using mixed finite element approximations was first studied in\cite {ATWZ96}. Shortly thereafter,  Woodward and Dawson in \cite{WCD00} studied  the expanded mixed finite element methods for a nonlinear parabolic equation modeling flow into variably saturated porous media. Recently,  Galerkin finite element method for a coupled nonlinear degenerate system of advection-diffusion equations were studied in \cite{F06,F07,FS95,FS04}. In their analysis, the Kirchhoff transformation is used to move the nonlinearity from coefficient $K$ to the gradient and thus simplifies the analysis of the equations. This transformation does not applicable for the equation \eqref{maineq}.  

In this paper,  we focus on the case of Degree Condition, see  \eqref{deg} in the next section, for the following reasons. 
First, it already covers the most commonly used  Forchheimer equations in practice, namely, the two-term, three-term and power laws.
Second, it takes advantage of the well-known Poincar\'e-Sobolev embeddings in our work.
Third, it makes clear our ideas and techniques without involving much more complicated technical details in case that the Degree Condition is not met (see \cite{HIKS1, K1}). 
%--------------------------------------------------------------------
For our degenerate equations, we combine the techniques in \cite {HI1,HI2,HIKS1, HK1,HK2,HKP1, NJ2000,TW06} and utilize the special structures of the equations to obtain the long-time stability and error estimates for the approximate solution in several norms of interest. Though the error estimates are not optimal order due to the lack of regularity of the solution, these results are obtained with the minimum regularity assumptions.
 %--------------------------------------------------------------------

 The paper is organized as follows.  
  In section \S \ref{Intrsec}, we introduce notations and some of relevant results. %in \cite{ABHI1,HI1,K1,NJ2000}.%--------------------------------------------------------------------
 In section \S \ref{GalerkinMethod}, we consider the semidiscrete finite element  Galerkin approximation and the implicit backward difference time discretization to the initial boundary value problem \eqref{rho:eq}. 
%--------------------------------------------------------------------
 In section \S \ref{Bsec}, we establish many estimates of the energy type norms for the approximate solution $\tilde \rho_h$.  In Theorems~\ref{bound-lq}, \ref{phderv}, \ref{Est4Grad}, the bounds for approximate solution, its time derivative and gradient vector are established for all time and time $t\to \infty$. The uniformly large time estimates, asymptotic estimates are in Theorems~\ref{UniGrad},\ref{Uni-rhot}.
  %--------------------------------------------------------------------
In section \S \ref{errSec},  we analyze two versions of the Galerkin finite element approximations: the continuous Galerkin method and  the discrete Galerkin method. Using the monotonicity properties of Forchheimer equation and the boundedness of the solution, the {\it priori} error estimates for all time, long time, are derived for the solution in $L^2$, $L^\infty$ and for the gradient vector in $L^{2-a}$. The main results are stated and proved in Theorems~\ref{longerr1}--\ref{longerr2}.
%--------------------------------------------------------------------
In section \S \ref{Num-result}, the results of a few numerical experiments using the Lagrange elements of order $1$ in the two-dimensions are reported.  These results support our theoretical analysis regarding convergence rates.  

%============================================================

\section{Notations and auxiliary results}\label{Intrsec}

Suppose that $\Omega$ is an open, bounded subset of $\mathbb{R}^d$, d=2,3,\ldots, with boundary $\Gamma$ smooth. Let $L^2(\Omega)$ be the set of square integrable functions on $\Omega$ and $( L^2(\Omega))^d$ the space of $d$-dimensional vectors with all the components in $L^2(\Omega)$.  We denote $(\cdot, \cdot)$ the inner product in either $L^2(\Omega)$ or $(L^2(\Omega))^d$.  The notation $\norm {\cdot}$ means scalar norm $\norm{\cdot}_{L^2(\Omega)}$ or vector norm $\norm{\cdot}_{(L^2(\Omega))^d}$ and $ \norm{\cdot}_{L^p}=\norm{\cdot}_{L^p(\Omega)}$ represents the standard  Lebesgue norm.
Notation $\norm{\cdot}_{L^p(L^q)} =\norm{\cdot}_{L^p(0,T; L^q(\Omega))}, 1\le p,q<\infty$ means the mixed Lebesgue norm while $\norm{\cdot}_{L^p(H^q)}=\norm{\cdot}_{L^p(0,T;H^q(\Omega))}, 1\le p,q<\infty$ stands for the mixed Sobolev-Lebesgue norm. 

For $1\le q\le +\infty$ and $m$ any nonnegative integer, let
\beqs
W^{m,q}(\Omega) = \big\{u\in L^q(\Omega), D^q u\in L^q(\Omega), |q|\le m \big\}
\eeqs 
denote a Sobolev space endowed with the norm 
\beqs
\norm{u}_{m,q} =
\left( \sum_{|i|\le m} \norm{D^i u}^q_{L^q(\Omega)} \right)^{\frac 1 q}.
\eeqs    
Define $H^m(\Omega)= W^{m,2}(\Omega)$ with the norm $\norm{\cdot}_m =\norm{\cdot }_{m,2}$. 

Throughout this paper, we use short hand notations, 
\beqs
\norm{\rho(t)} = \norm{ \rho(\cdot, t)}_{L^2(\Omega)}, \forall t\ge 0 \quad \text{
 and } \quad \rho^0(\cdot) =  \rho(\cdot,0).
 \eeqs
 
  Our calculations frequently use the following exponents
\beq\label{a-const }
   a=\frac{\alpha_N}{\alpha_N+1}= \frac{\deg (g)}{\deg (g)+1},
  \eeq
\beq\label{MLGdef}
  \beta =2-a,\quad \lambda = \frac {2-a}{1-a}=\frac{\beta}{\beta-1},\quad \gamma=\frac{a}{2-a}=\frac a \beta.
  \eeq  

The letters $C, C_0, C_1,C_2\ldots$ represent for the positive generic constants.  Their values  depend on exponents, coefficients of polynomial  $g$,  the spatial dimension $d$ and domain $\Omega$, independent of the initial and boundary data, size of mesh and time step. These constants may be different from place to place. 

\textbf{Degree Condition:} All of the following are equivalent conditions:
\beq\label{deg}     \deg(g)\le \frac{4}{d-2},\quad  a\le \frac 4{d+2},\quad  2\le \beta^*,\quad  2-a\ge \frac{2d}{d+2}.
\eeq
Here $\beta^*$ is the Sobolev conjugate of $\beta$, given by
$\beta^* = \frac{\beta d}{\beta-d}.$ 

Throughout this paper, we assume the Degree Condition. Whenever this condition is met,
the Sobolev space $W^{1,\beta}(\Omega)$ is continuously embedded into $L^2(\Omega)$.

%%%%%%%%%%%%%%%%%%%%%%%%%%%%%%%%%%%%%%%%%%%%%%%%%%%%%%%%%%%%%%%%%%%%%
\begin{lemma}[cf. \cite{ABHI1, HI1}, Lemma 2.1]
 The function $K(\xi)$ has the following properties
 \begin{enumerate}
\item [i.] $K: [0,\infty)\to (0,a_0^{-1}]$ and it decreases in $\xi,$  

\item[ii.] For any $n\ge 1$, the function $K(\xi)\xi^n$ increasing and $K(\xi)\xi^n\ge 0$

\item[iii.]  Type of degeneracy  \beq\label{i:ineq1}  \frac{c_1}{(1+\xi)^a}\leq K(\xi)\leq \frac{c_2}{(1+\xi)^a},  \eeq

\item[iv.]  For all $n\ge 1,\delta>0,$ 
\beq\label{i:ineq2} 
c_3\left(\frac{\delta}{1+\delta} \right)^a (\xi^{n-a}-\delta^{n-a})\leq K(\xi)\xi^n\leq c_2\xi^{n-a}, 
\eeq
In particular, when $n=2$, $\delta=1$ 
\beq\label{iineq2} 
2^{-a}c_3 (\xi^{2-a}-1)\leq K(\xi)\xi^2\leq c_2\xi^{2-a}, 
\eeq

 \item[v.]  Relation with its derivative \beq\label{i:ineq3} -aK(\xi)\leq K'(\xi)\xi\leq 0, \eeq
  where $c_1, c_2, c_3$ are positive constants depending on $\Omega$ and $g$. 
  \end{enumerate}
\end{lemma}
We define 
\beq\label{Hdef}
H(\xi)=\int_0^{\xi^2} K(\sqrt{s}) dx, \text{~for~} \xi\geq 0. 
\eeq
The function $H(\xi)$ is compared with $\xi$ and $K(\xi)$ by
\beq\label{i:ineq4}
K(\xi)\xi^2 \leq H(\xi)\leq 2K(\xi)\xi^2. 
\eeq

%We define  for $q\ge 2,$
%\beq\label{Hdefq}
%H_q(\rho)=\int_0^{|\nabla \rho|^q} K(\xi^{\frac 1q}) d\xi, \text{~for~} \xi\geq 0. 
%\eew
%The function $H(\xi)$ can compare with $\xi$ and $K(\xi)$ by
%\beq\label{i:ineq4}
%K(|\nabla \rho|)|\nabla \rho|^q \leq H_q(\rho)\leq qK(|\nabla \rho|)|\nabla \rho|^q. 
%\eeq
%\begin{proof}
%The function $K(\xi)\xi^n, n\ge 1$ increases,    
%\beqs
%H_q(\rho)=q\int_0^{|\nabla \rho|}  K(s)s^{q-1} ds \le q\int_0^{|\nabla \rho|}  K(|\nabla \rho|)|\nabla \rho|^{q-1} ds = qK(|\nabla \rho|)|\nabla \rho|^{q}
%\eeqs
%and because $K(\cdot)$ decreases function,  
%\beqs
%H_q(\rho)\ge \int_0^{|\nabla \rho|^q} K(|\nabla \rho|) d\xi = K(|\nabla \rho|)|\nabla \rho|^q.
%\eeqs
%\end{proof}

%as a consequence of \eqref{i:ineq2} and \eqref{i:ineq4} we have,
%\beq
%C(\Phi^{\beta}-1) \leq H(\Phi) \leq 2C\Phi^{\beta}. \label{i:ineq5}
%\eeq 
%==================================================================%
For the monotonicity and continuity of the differential operator in \eqref{maineq}, we have the following results. 
\begin{lemma}[cf. \cite{HI1,ABHI1}, Lemma 5.2, Lemma III.11] The following statements hold 
\begin{itemize}
\item [i.] For all $y, y' \in \mathbb{R}^d$, 
\beq\label{Qineq}
\big(K(|y'|)y' -K(|y|)y \big)\cdot(y'-y)\geq (\beta-1)K( \max\{ |y|, |y'|\} )|y' -y|^2 .
\eeq      

 \item [ii.] For the vector functions $s_1, s_2$, there is a positive constant $c_4(\Omega, d,g)$ such that
\beq\label{Mono}
 \big( K(|s_1|)s_1-K(|s_2|)s_2,s_1 -s_2\big)\geq c_4\omega\norm{s_1-s_2}_{0,\beta}^2,
\eeq

where 
$$\omega =\left(1+ \max\left\{\norm{s_1}_{0,\beta} ;  \norm{s_2}_{0,\beta} \right\}\right)^{-a}. $$
\end{itemize}
\end{lemma}
%======================================
\begin{lemma}[cf. \cite{K1}, Lemma 2.4] \label{Lips}  For all vector $y, y' \in \mathbb{R}^d$. There exists a positive constant $c_5$ depending on polynomial  $g$, the spatial dimension $d$ and domain $\Omega$ such that 
\beq\label{Lipchitz}
   \left|K(|y'|)y' -K(|y|)y\right| \leq c_5|y' -y|.
\eeq  
\end{lemma}
The following Poincar\'e-Sobolev inequality with weight is used in our estimate later.      
 \begin{lemma}[cf. \cite{HI1} Lemma 2.4] Let $\xi(x)\ge 0$ be defined on $\Omega$. Then for any function $u(x)$ vanishing on the boundary $\Gamma$, there is a positive constant $c_6(\Omega, d,N,g)$ such that
\beq\label{PSineq}
\norm{u}_{0,\beta^*}^2 \le c_6 \norm{K^{\frac 12}(\xi) \nabla u}^2\left(1+ \norm{K^{\frac 12}(\xi) \xi}^2\right)^{\gamma}.
\eeq

Under degree condition (DC), i.e. $\deg(g) \le \frac{4}{d-2}$, we have
\beq\label{embL2}
\norm{u}^2 \le c_6 \norm{K^{\frac 12}(\xi) \nabla u}^2\left(1+ \norm{K^{\frac 12}(\xi) \xi}^2\right)^{\gamma}.
\eeq
\end{lemma} 
\begin{definition}\label{Env}
Given $f(t)$ defined on an interval $I\subset\mathbb R$. A function $F(t)$ is called an
(upper) envelop of $f(t)$ on $I$ if $F(t) \ge f(t)$ for all $t \in I$. We denote by $Env(f)$ a
continuous, increasing envelop function of $f(t)$.
\end{definition}
%==================================================================%

We state several Gronwall-type inequalities which are useful in our estimates analysis.      
 
 \begin{lemma}[cf. \cite{HIKS1}, Lemma 2.7] \label{ODE2} 
Let $\theta>0$ and let $y(t)\ge 0, h(t)>0, f(t)\ge 0$ be continuous functions on $[0,\infty)$ that satisfy
 \beqs
 y'(t) +h(t)y(t)^\theta \le f(t),\quad \text{for all } t>0.
 \eeqs
Then
 \beq\label{ubode}
 y(t)\le y(0)+\left[Env\left(\frac{f(t)}{h(t)}\right)\right]^\frac{1}{\theta}, \quad \text{ for all } t\ge 0.
 \eeq
If $\int_0^\infty h(t)dt=\infty$ then
 \beq\label{ulode}
 \limsup_{t\rightarrow\infty} y(t)\le \limsup_{t\rightarrow\infty} \left[\frac {f(t)}{h(t)}\right]^\frac{1}{\theta}.
\eeq
 \end{lemma}
 
\begin{lemma}  [cf. \cite{NJ2000} Lemma 2.4]  Assume $f\ge 0,$ $h,\theta>0$ and $y(t)\ge 0$ be a continuous function on $[0,\infty)$ satisfying 
\beqs
y'(t) + h y(t)^\theta \le f, \quad \text{ for all } t\ge 0 
\eeqs 
then 
\beqs
y(t)\le \max\left\{ y(0), \left(\frac{f}{h}\right)^{1/\theta}  \right\}.
\eeqs
\end{lemma}
%===============================
\begin{lemma}[cf. \cite{NJ2000} Lemma 2.5, Discrete Gronwall's inequality]  \label{DGronwall}  Assume $f\ge 0,$ $h>0$, $ \theta>0$, $\Delta t>0$ and the sequence $\{y_n\}_{n=1}^\infty$ nonnegative satisfying 
\beqs
\frac{y_{n} -y_{n-1}}{\Delta t} + h y_n^\theta \le f, \quad \text{ for all } n=1,2,\ldots 
\eeqs 
then 
\beqs
y_n\le \max\left\{ y_0, \left(\frac{f}{h}\right)^{1/\theta}  \right\}.
\eeqs
\end{lemma}

%==================================================================%

%%%%%%%%%%%%%%%%%%%%%%%%%%%%%%%%%%%%%%%%%%%%%%%%%%%%%%%%%%%%%%%%%%%%%
%\myclearpage    
\section{The Galerkin finite element method}\label{GalerkinMethod}
%Our aim is to study equation \eqref{maineq} for density of slightly compressible fluids in bounded
%domain in porous media. The fluid flows are subject to some conditions on the boundary. 
We consider the initial boundary value problem associated with \eqref{maineq} ,
\beq\label{rho:eq}
\begin{cases}
\rho_ t - \nabla \cdot (K (|\nabla \rho|)\nabla \rho  ) =f, &\text {in }  \Omega\times \R_+,\\
\rho(x,0)=\rho^0(x), &\text {in } \Omega,\\
\rho(x,t)=\psi(x,t),  &\text{on } \Gamma \times  \R_+,
\end{cases}
\eeq 
where $\rho^0(x)$ and $\psi(x,t)$ are given initial and boundary data, respectively.  

To deal with the non-homogeneous boundary condition, we extend the Dirichlet boundary data  from boundary $\Gamma$ to the whole domain $\Omega$, see \cite{HI1,JK95, MC70}. Let $\phi(x,t)$ be such an extension.  
Let $\rhot = \rho-\phi$. Then $\rhot (x,t)=0 \text{ ~on~} \Gamma \times \Rplus.$ System  \eqref{rho:eq} is rewritten as  
\beq\label{rhoeq}
\begin{cases}
\rhot_ t - \nabla \cdot (K (|\nabla \rho|)\nabla \rho  ) =\tilde f, &\text {in }  \Omega\times \R_+,\\
\rhot(x,0)=\rhot^0(x), &\text {in } \Omega,\\
\rhot(x,t)=0,  &\text{on } \Gamma \times  \R_+,
\end{cases}
\eeq 
where  $ \rhot^0 = \rho^0(x) -\phi(x,0) $ and $\tilde f= f-\phi_t$. 

The variational formulation of \eqref{rho:eq} is defined as the follows: Find $\rhot :\Rplus \rightarrow W \equiv H_0^1  $ such that 
\beq\label{weakform}
(\rhot_t, w) + (K(|\nabla \rho |)\nabla \rho,\nabla w) = (\tilde f, w),  \quad \forall w\in H_0^1(\Omega)
\eeq 
with $\rhot(x,0)=\rhot^0(x).$ 

\vspace{0.2cm}
%{\bf Finite element space.} We will work with the standard Galerkin finite element method. In particular, we 
Let  $\{\mathcal T_h\}_h$ be a family of globally quasiuniform triangulations of $\Omega$ with $h$ being the maximum diameter of the element. Let $W_h$ be the space of discontinuous piecewise polynomials of degree $r\ge 0$ over  $\mathcal T_h$.     
It is frequently valuable to decompose the analysis of the convergence of finite element methods by passing through a projection of the solution of the differential problem into the finite element space. 

 We use the standard $L^2$-projection operator,  see \cite{Ciarlet78},  
$\pi: H^{1}(\Omega) \rightarrow W_h$,   satisfying
\begin{align*}
( \pi w ,   v_h ) = ( w ,   v_h ), \quad &\forall w\in W, v_h \in W_h.
\end{align*}
This projection has well-known approximation properties, see~\cite{BF91,JT81,BPS02}.
\begin{itemize}
%(i) $\norm{\pi w}\le \norm{w}$ holds for all $w\in L^2(\Omega)$.
\item[i.]  For all $w\in H^s(\Omega), s\in \{0,1\}$, there  is a positive constant $C_0$ such that
\beq\label{L2proj0}
\norm{\pi w}_{s}\le C_0 \norm{w}_{s}.   
\eeq

\item[ii.] There exists a positive constant $C_1$ such that
\beq\label{prjpi}
\norm{\pi w - w }_{0,q} \leq C_1 h^m \norm{w}_{m,q} 
\eeq
for all $w \in W^{m,q}(\Omega)$,  $0\le m \le r+1, 1\le q \le \infty$.  

%Here $\norm {\cdot}_{m,q}$ denotes a standard norm in Sobolev space $W^{m,q}(\Omega)$. In short hand, when $q=2$ we write \eqref{prjpi} as   
%\beqs
%\norm{\pi w - w } \leq C_1 h^m \norm{w}_{m}. 
%\eeqs
\end{itemize}

 The semidiscrete formulation of~\eqref{weakform} can read as follows: Find $\rhot_h =\rho_h-\pi \phi:\R_+\rightarrow W_h$ such that
\beq\label{semidiscreteform}
  (\rhot_{h,t},w_h)+(K(|\nabla \rho_h|)\nabla \rho_h,\nabla w_h) = (\ft, w_h),  \quad \forall w_h\in W_h
\eeq 
with initial data $\rhot_h^0=\pi \rhot^0(x)$.

We use backward Euler for time-difference discretization. Let $\{t_i\}_{i=1}^\infty$ be the uniform partition of $\Rplus$ with $t_i=i\Dt$, for time step $\Dt >0$. We define $\varphi^n = \varphi(\cdot, t_n)$. 

The discrete time Galerkin finite element approximation to \eqref{weakform} is defined as follows:  
Find $ \rhot_h^n\in W_h$, $n=1,2,\dots$  such that 
\beq\label{fullydiscreteform}
\Big( \frac{ \rhot_h^n - \rhot_h^{n-1}}{\Delta t }, w_h\Big) +  \big(K(|\nabla \rho_h^n|)\nabla \rho_h^n, \nabla w_h\big) =(\ft^n, w_h ), \quad \forall w_h\in W_h.
\eeq

The initial data is chosen by
$
\rhot_h^0(x)=\pi \rhot^0(x). 
$

\section{A priori estimate for solutions}\label{Bsec}
We study the  equations \eqref{weakform}, and \eqref{semidiscreteform} for the density with fixed functions $g(s)$ in \eqref{eq1} and \eqref{eq2}. 
Therefore, the exponents $\alpha_i$ and coefficients $a_i$ are all fixed, and so are the functions $K(\xi)$, $H(\xi)$  in \eqref{Kdef}, \eqref{Hdef}.  

With the properties \eqref{i:ineq1}, \eqref{i:ineq2}, \eqref{i:ineq3}, the monotonicity \eqref{Qineq}, and by
classical theory of monotone operators \cite{MR0259693,s97,z90}, the authors in \cite {HIKS1} proved the global existence and uniqueness of the weak solution of the equation \eqref{weakform}.  For  the { \it priori } estimates, we assume that the weak solution is a sufficient regularity in both  $x$ and $t$ variables. 
Hereafter, we only consider solutions $\rhot(x; t)$ that satisfy $\rhot\in C^2(\bar\Omega\times \R_+ )$ and $\rhot,\nabla \rhot \in C( \bar\Omega \times \R_+)$.
%===========================
\begin{theorem}\label{bound-lq}
Let $\tilde \rho_h$ be a solution of the problem \eqref{semidiscreteform}. Then, there exists  a positive constant $C$ such that for all $t>0,$
\beq\label{res1}
\norm{\rhot_h (t)}\le  C\left(1+\norm{\rhot^0 } + \left(Env\, g(t) \right) ^{\frac 1 \beta}\right),
\eeq
where 
\beq\label{gdef}
g(t)= \norm{\ft(t)}^{\lambda} + \norm{ \phi(t)}_{1}^2.
\eeq

Furthermore,
 \beq\label{limsup:rhot}
\limsup_{t\to\infty}\norm{\rhot_h(t)}^2\le  C \left(1+ \limsup_{t\to\infty} g(t)  \right)^\frac{2}{\beta}.
\eeq
If 
\beq\label{small:atinf}
\limsup_{t\to \infty} \norm{\ft(t)} = \limsup_{t\to \infty}\norm{\phi(t)}_{1}=0
\eeq
  then
\beq\label{res1a}
\limsup_{t\to\infty} \norm{\rhot_h (t)}^2=0.
\eeq
\end{theorem}

\begin{proof}
 Selecting $w_h= \rhot_h$ in \eqref{semidiscreteform}, we obtain
\beq\label{dpalpha}
 \begin{split}
\frac{1}{2}\frac{d }{dt}\norm{\rhot_h}^2 +\norm{K^{\frac{1}{2}}(|\nabla \rho_h|)\nabla \rho_h}^2= (K|\nabla \rho_h|)\nabla \rho_h,\nabla \pi\phi  )
  + \intd{\ft,\rhot_h}.
\end{split}
\eeq
%We will estimate the last three integrals of \eqref{dpalpha}. 
By Cauchy's inequality and the upper boundedness of $K(\cdot)$, we have
\beq\label{term1}
\begin{aligned}
(K|\nabla \rho_h|)\nabla \rho_h,\nabla \pi\phi  )&\le \frac 14 \norm{K^{\frac{1}{2}}|\nabla \rho_h|)\nabla \rho_h}^2 + C\norm{ \nabla \pi\phi}^2    .
\end{aligned}
\eeq
Using H\"older's inequality and \eqref{embL2} give 
\beqs
(\ft, \rhot_h )\le  \norm{\ft}\norm{\rhot_h}\le C \norm{\ft}\norm{K^{\frac 12}(|\nabla \rho_h| ) \nabla \rhot_h}\left(1+ \norm{K^{\frac 12}(|\nabla\rho_h|) \nabla\rho_h}^2\right)^{\frac \gamma 2}.
\eeqs
Thanks to triangle inequality, $(a+b)^\gamma\le 2^\gamma(a^\gamma+b^\gamma),\ \forall a,b\ge 0$, and Young's inequality we find that  
\begin{align*}
 &\norm{K^{\frac 12}(|\nabla \rho_h| ) \nabla \rhot_h}\left(1+ \norm{K^{\frac 12}(|\nabla\rho_h|) \nabla\rho_h}^2\right)^{\frac \gamma 2}\\
&\qquad\le C \left( \norm{K^{\frac 12}(|\nabla \rho_h| ) \nabla \rho_h} + \norm{\nabla \pi\phi}\right) 
\left(1+ \norm{K^{\frac 12}(|\nabla\rho_h|) \nabla\rho_h}^\gamma  \right)\\
&\qquad\le C \left\{ \norm{K^{\frac 12}(|\nabla \rho_h| ) \nabla \rho_h}+ \norm{K^{\frac 12}(|\nabla\rho_h|) \nabla\rho_h}^{\gamma+1}+\norm{\nabla \pi\phi}\norm{K^{\frac 12}(|\nabla \rho_h| ) \nabla \rho_h}^{\gamma} + \norm{\nabla \pi\phi} \right \}\\
&\qquad\le C \left\{ 1+ \norm{K^{\frac 12}(|\nabla \rho_h|) \nabla\rho_h}^{\gamma+1} +\norm{\nabla \pi\phi}^{\gamma+1} \right\}.
\end{align*}
This and Young's inequality applying for $\norm{\ft}\norm{\nabla \pi\phi}^{\gamma+1}$ with the exponents $\lambda$, $\frac{\lambda}{\lambda-1}$ yield
\beq\label{term2}
\begin{split}
(\ft, \rhot_h )&\le C \norm{\ft}\left\{ 1+ \norm{K^{\frac 12}(|\nabla \rho_h|) \nabla\rho_h}^{\gamma+1} +\norm{\nabla \pi\phi}^{\gamma+1} \right\}\\
&\le C \norm{\ft}+ C\norm{\ft}^{\lambda} +  \frac 14\norm{K^{\frac 12}(|\nabla \rho_h|) \nabla \rho_h}^2 +C\norm{\nabla \pi\phi}^2 .
\end{split}
\eeq
Combining \eqref{term1}, \eqref{term2} and \eqref{dpalpha} gives
\beq\label{diff-neq}
  \frac{d }{dt}\norm{\rhot_h}^2+\norm{K^{\frac{1}{2}}(|\nabla \rho_h|)\nabla \rho_h}^2
% &\le C\norm{ \nabla \pi\phi}^2 + C \norm{\ft}+ C\norm{\ft}^{\lambda}  +C\norm{\nabla \pi\phi}^2\\
 \le C\left(\norm{ \nabla \pi\phi}^2 +  \norm{\ft}+ \norm{\ft}^{\lambda}\right).
\eeq
%The last inequality in \eqref{diff-neq} is followed by applying Young's inequality for $\norm{\ft}\norm{\nabla \pi\phi}^{\gamma+1}$ with the exponents $\lambda$, $\frac{\lambda}{\lambda-1}$.     
We have from \eqref{i:ineq2} that
\beqs
c_3\left(\frac{\delta}{1+\delta} \right)^a\left(\norm{\nabla \rho_h}_{0,\beta}^\beta -\delta^{\beta}\right)\le \norm{K^{\frac{1}{2}}(|\nabla \rho_h|)\nabla \rho_h}^2.
\eeqs
In virtue of the inequality $(a+b)^m \le 2^{m-1}(a^m+b^m),\forall a,b\ge 0, m\ge 1$,  
\beqs
\norm{\nabla \rhot_h}_{0,\beta}^\beta \le 2^{\beta-1}\left( \norm{\nabla \rho_h}_{0,\beta}^\beta +\norm{\nabla \pi\phi}_{0,\beta}^\beta \right).
\eeqs
Combining the two above inequalities gives
\beq\label{trhoEst}
c_3\left(\frac{\delta}{1+\delta} \right)^a\left(2^{1-\beta}\norm{\nabla \rhot_h}_{0,\beta}^\beta  -\norm{\nabla \pi\phi}_{0,\beta}^\beta- \delta^{\beta}\right)\le \norm{K^{\frac{1}{2}}(|\nabla \rho_h|)\nabla \rho_h}^2.
\eeq

Under the condition on the degree of the polynomial $g$, i.e under (DC), using the  Poincar\'e-Sobolev inequality, we obtain
\beq\label{PSIneq}
\norm{\rhot_h} \le C\norm{\rhot_h}_{0,\beta^*}\le   C_p \norm{\nabla \rhot_h }_{0,\beta} .
\eeq 
It follows from  \eqref{diff-neq}, \eqref{trhoEst} and \eqref{PSIneq} that
\begin{align*}
\frac{d }{dt}\norm{\rhot_h}^2&+ c_32^{1-\beta}C_p^{-\beta}\left(\frac{\delta}{1+\delta} \right)^a \norm{\rhot_h}^\beta \\
& \le  C\left(\norm{\ft}^{\lambda} + \norm{\ft} + \norm{\nabla \pi\phi}^2\right)
 +C\left(\frac{\delta}{1+\delta} \right)^a \norm{\nabla \pi\phi}_{0,\beta}^\beta +C\left(\frac{\delta}{1+\delta} \right)^a\delta^{\beta}\\
& \le  C\left(\norm{\ft}^{\lambda} + \norm{\ft}+ \norm{\nabla \pi\phi}^2
 +\norm{\nabla \pi\phi}_{0,\beta}^{\beta}\right) +C\left(\frac{\delta}{1+\delta} \right)^a\delta^{\beta}.
\end{align*}
According to the Gronwall's inequality in Lemma \ref{ODE2} with $\delta=1$,  we find that 
\beq\label{}
\begin{split}
\norm{\rhot_h}^2&\le  C\norm{\rhot_h^0}^2 + C  \left [ Env\Big(\norm{\ft}^{\lambda} +\norm{\ft}+ \norm{\nabla \pi\phi}^2
 +\norm{\nabla \pi\phi}^{\beta} +1 \Big)  \right]^{\frac{2}{\beta}}\\
&\le  C\norm{\rhot_h^0}^2 + C  \left [ Env\Big( \norm{\ft}^{\lambda} + \norm{\nabla \pi\phi}^2\Big) +1 \right]^{\frac{2}{\beta}}.
\end{split}
\eeq
This and the stability of $L^2$- projection \eqref{L2proj0} show that \eqref{res1} holds true.   

Due to $\int_0^\infty c_32^{1-\beta}C_p^{-\beta} dt =\infty$, it follows from \eqref{ulode} that,  
  \beq
 \begin{split}
\limsup_{t\to\infty}\norm{\rhot_h}^2&\le  C \left [ \left(\frac{\delta}{1+\delta} \right)^{-a} \limsup_{t\to\infty}\Big(\norm{\ft}^{\lambda} + \norm{\ft}+ \norm{\nabla \pi\phi}^2
 +\norm{\nabla \pi\phi}^{\beta}\Big) +\delta^{\beta}  \right]^{\frac{2}{\beta}}\\
 &\le  C \left [ \left(\frac{\delta}{1+\delta} \right)^{-a} \limsup_{t\to\infty}\Big(\norm{\ft}^{\lambda} + \norm{\ft}+ \norm{\phi}_{1}^2
 +\norm{\phi}_{1}^{\beta}\Big) +\delta^{\beta}  \right]^{\frac{2}{\beta}}.
\end{split}
\eeq

Therefore,  if we choose $\delta=1$ in the previous inequality then we obtain \eqref{limsup:rhot}.
 
Under assume \eqref{small:atinf} then   
$$
 \limsup_{t\to\infty}\Big(\norm{\ft}^{\lambda} + \norm{\ft}+ \norm{\phi}_{1}^2
 +\norm{\phi}_{1}^{\beta}\Big)=0.
$$
Hence,
$$
\limsup_{t\to\infty}\norm{\rhot_h}^2 \le C\delta^2. 
$$
Letting $\delta\to 0$, we obtain  
$$
\limsup_{t\to\infty}\norm{\rhot_h}^2 =0.
$$
The proof is complete.  
\end{proof}

Since the equation \eqref{semidiscreteform} can be interpreted as the finite system of ordinary differential equations in the coefficients of $\rho_h$ with respect to basis of $W_h$. The stability estimate \eqref{res1} suffices to establish the local existence of $\rho_h(t)$ for all $t\in \R_+.$   
The uniqueness of the approximation solution comes from the monotonicity of  operator, see \cite{HI1}.
%\subsection{Estimate for gradient }

Now we derive an estimate for the gradient of approximated solution.
 
\begin{theorem}\label{Est4Grad} Assume $\tilde \rho_h$ a solution to the problem \eqref{semidiscreteform}. Then, there exists a positive constant $C$ such that for all $t\ge 0$, 
\beq\label{res2}
 \norm{\nabla \rho_h(t)}_{0,\beta}^{\beta}  \le C \Aa(t),
\eeq
where 
\beq\label{Mdef}
\begin{split}
\Aa(t) &= 1+\norm{ \rho^0}_1^2+ \norm{\phi^0}^2 +\int_0^t e^{-\frac 1 2 (t-s)} \left(  \norm{\phi_t(s)  }_{1}^2 +  (Env\, g(s))^{\frac 2 {\beta}}  \right)ds.
\end{split}
\eeq

Furthermore,
\beq\label{limsupGrad}
 \limsup_{t\to \infty} \norm{\nabla \rho_h(t)}_{0,\beta}^\beta \le C\left(1+\limsup_{t\to\infty}\norm{\phi_t(t) }_{1}^2+(\limsup_{t\to\infty} g(t))^{\frac{2}{\beta}}\right).
  \eeq 
If 
\beq\label{small:atinf2}
\limsup_{t\to \infty} \norm{\ft(t)} = \limsup_{t\to \infty}\norm{\phi(t)}_{1}= \limsup_{t\to \infty}\norm{\phi_t(t)}_{1}= 0
\eeq
then 
 \beq\label{Grad:small}
  \limsup_{t\to \infty}\norm{\nabla \rho_h(t)}_{0,\beta}  =0.
 \eeq
\end{theorem}
%=====================
\begin{proof}
Choosing $w_h=\rhot_{h,t}$ in \eqref{semidiscreteform} leads to 
\beq\label{rhot}
\norm{ \rhot_{h,t}}^2 +\frac 1 2 \ddt \int_\Omega H(x,t) dx  
=(K(|\nabla \rho_h|)\nabla\rho_h, \nabla \pi\phi_t ) +(\ft, \rhot_{h,t}).
\eeq
where $H(x,t)=H(|\nabla \rho_h(x,t)| )$ is defined in \eqref{Hdef}.
% Note that from \eqref{dpalpha} we have
%  \beq
%\frac{1}{2}\frac{d }{dt}\norm{\rhot_h}^2 +\norm{K^{\frac{1}{2}}(|\nabla \rho_h|)\nabla \rho_h}^2= (K|\nabla \rho_h|)\nabla \rho_h,\nabla \pi\phi  )
%  + \intd{\ft,\rhot_h}.
%\eeq
 
Adding the two equations \eqref{rhot} and \eqref{dpalpha} gives  
\begin{multline*}
\norm{ \rhot_{h,t}}^2  + \norm{K^{\frac{1}{2}}(|\nabla \rho_h|)\nabla \rho_h}^2
 +\frac 12 \ddt  \Big(\int_\Omega H(x,t)dx  + \norm{\rhot_h}^2\Big) \\
   =   (K|\nabla \rho_h|)\nabla \rho_h,\nabla \pi\phi+\nabla \pi\phi_t  )+  \intd{\ft,\rhot_h + \rhot_{h,t}}.
\end{multline*}

Using Cauchy's inequality gives
\begin{multline*}
\norm{ \rhot_{h,t}}^2  + \norm{K^{\frac{1}{2}}(|\nabla \rho_h|)\nabla \rho_h}^2
+\frac 12 \ddt  \Big(\int_\Omega H(x,t)dx  + \norm{\rhot_h}^2\Big)\\
   \le \frac 12  \norm{K^{\frac{1}{2}}(|\nabla \rho_h|)\nabla \rho_h}^2 
  +\frac  12 \norm{ \nabla \pi\phi +\nabla \pi\phi_t  }^2+ 2\norm{\ft}^2 +\frac 14 \norm{\rhot_h}^2+\frac 1 4 \norm{\rhot_{h,t}}^2, 
\end{multline*}
which implies 
\begin{multline}\label{sdq}
\frac 3 4 \norm{ \rhot_{h,t}}^2  + \frac 12\norm{K^{\frac{1}{2}}(|\nabla \rho_h|)\nabla \rho_h}^2
+\frac 12 \ddt  \Big(\int_\Omega H(x,t)dx  + \norm{\rhot_h}^2\Big)\\
   \le  \norm{ \nabla \pi\phi}^2 +\norm{\nabla \pi\phi_t  }^2+ 2\norm{\ft}^2 +\frac 14 \norm{ \rhot_h}^2.
\end{multline} 
Now using \eqref{i:ineq4}, we find that $$\norm{K^{\frac{1}{2}}(|\nabla \rho_h|)\nabla \rho_h}^2\ge \frac 1 2 \int_\Omega H(x,t) dx.$$ 
This and \eqref{sdq} show that 
\beq\label{key:ineq2}
\begin{aligned}
&\frac 3 4 \norm{ \rhot_{h,t} }^2  + \frac 1 4 \int_\Omega H(x,t) dx +\frac 12 \ddt  \Big(\int_\Omega H(x,t)dx  + \norm{\rhot_h}^2\Big)\\
&\hspace{4cm}\le  \norm{ \nabla \pi\phi }^2+\norm{\nabla \pi\phi_t  }^2+ 2\norm{\ft}^2 +\frac 14 \norm{\rhot_h}^2 .
\end{aligned}
\eeq

Note that $ \frac 12 \frac d{dt}\norm{\rhot_h}^2 = (\rhot_h, \rhot_{h,t})$, again using Cauchy's inequality leads to   
 \begin{multline*}
\frac 34 \norm{ \rhot_{h,t}}^2  +\frac 1 4\int_\Omega H(x,t) dx 
 +\frac 12 \ddt \int_\Omega H(x,t)dx \\
   \le  \norm{ \nabla \pi\phi }^2+\norm{\nabla \pi\phi_t  }^2+ 2\norm{\ft}^2 +\frac 14 \norm{\rhot_h}^2 +\frac 12\norm{\rhot_h}^2+ \frac 12\norm{\rhot_{h,t}}^2.
\end{multline*}
Hence,
\beq\label{keyest2}
\begin{split}
\frac 14 \norm{ \rhot_{h,t}}^2  +\frac 1 4\int_\Omega H(x,t) dx 
 &+\frac 12 \ddt \int_\Omega H(x,t)dx \\
 &\le   \norm{ \nabla \pi\phi}^2+\norm{\nabla \pi\phi_t  }^2 + 2\norm{\ft}^2 + \frac 3 4\norm{ \rhot_h}^2 .
\end{split}
\eeq
This gives 
\begin{multline}\label{Grad-tder-eq}
\ddt \int_\Omega H(x,t) dx+\frac 1 2 \left(\norm{ \rhot_{h,t}}^2+ \int_\Omega H(x,t) dx\right) \le  4\left( \norm{ \nabla \pi\phi}^2+\norm{\nabla \pi\phi_t  }^2 + \norm{\ft}^2 + \norm{ \rhot_h}^2 \right)\\
 \le  C\left( \norm{ \phi}_{1}^2+\norm{\phi_t  }_{1}^2 + \norm{\ft}^2 + \norm{ \rhot_h}^2 \right).
\end{multline}
Dropping the nonnegative term on the left hand side in \eqref{Grad-tder-eq}, using \eqref{res1}, we obtain 
\beq\label{HdiffEq}
\begin{aligned}
  \ddt \int_\Omega H(x,t) dx+ \frac 12  \int_\Omega H(x,t) dx
 &\le  C\left( \norm{\phi}_{1}^2+\norm{\phi_t  }_{1}^2 + \norm{\ft}^2 + \norm{\rhot_h}^2 \right)\\
  &\le  C\left(\norm{\rhot^0}^2+ \norm{\phi_t  }_{1}^2 +  (Env(g(t))^{\frac{2}{\beta}} +1 \right).
\end{aligned}
\eeq
The last inequality is obtained by using the Young's inequality then absorbing $\norm{\ft}^\lambda $ and $\norm{\phi}_1
^2$ to $Env (g(t))^{\frac{2}{\beta}}$. 
  
It is followed by applying Gronwall's inequality to \eqref{HdiffEq} that
\beq\label{Diffineq3}
\begin{split}
\int_\Omega H(x,t) dx & \le e^{-\frac 1 2 t}\int_\Omega H(x,0) dx\\
 &\quad + C\int_0^t e^{-\frac 1 2 (t-s)} \left(\norm{\rhot^0}^2+ \norm{\phi_t  }_{1}^2 +  (Env \, g(t))^{\frac{2}{\beta}} +1 \right) ds.
\end{split}
\eeq
Due to \eqref{i:ineq4} and \eqref{i:ineq2},
\beqs
\begin{split}
\norm{\nabla \rho_h}_{0,\beta}^\beta & \le Ce^{-\frac 1 2 t}\norm{\nabla\rho_h^0}_{0,\beta}^\beta\\
 &\quad + C\int_0^t e^{-\frac 1 2 (t-s)} \left(\norm{\rhot^0}^2+ \norm{\phi_t  }_{1}^2 +  (Env\, g(t))^{\frac{2}{\beta}} +1 \right) ds+C\\
 &\le C+C\norm{\rho^0}_1^2+C\norm{\rhot^0}^2 + C\int_0^t e^{-\frac 1 2 (t-s)} \left(\norm{\phi_t  }_{1}^2 +  (Env\, g(t))^{\frac{2}{\beta}}\right) ds.
\end{split}
\eeqs
This proves \eqref{res2}.

Dropping the nonnegative term on the left hand side of \eqref{Grad-tder-eq}  and using \eqref{ulode} to \eqref{Grad-tder-eq}, we find that  
\beq\label{a}
\begin{split}
\limsup_{t\to\infty}\int_\Omega H(x,t) dx  &\le C  \left[\limsup_{t\to\infty}\left(\norm{\phi}_{1}^2+\norm{\phi_t}_{1}^2 + \norm{\ft}^2\right) +\limsup_{t\to\infty} \norm{ \rhot_h}^2 \right]\\
&\le C  \left[\limsup_{t\to\infty}\left( 1+ g(t)+\norm{\phi_t}_{1}^2\right) +(\limsup_{t\to\infty} g(t))^\frac{2}{\beta}+1 \right]\\
&\le C  \left[\limsup_{t\to\infty}\norm{\phi_t}_{1}^2 +(\limsup_{t\to\infty} g(t))^\frac{2}{\beta}+1 \right].
\end{split}
\eeq
From \eqref{i:ineq2}  and \eqref{i:ineq4} we have 
\beq\label{b}
c_3\left( \frac {\delta}{1+\delta} \right)^a\left( \norm{\nabla \rho_h}_{0,\beta}^\beta- \delta^\beta\right)\le\norm{K^{\frac{1}{2}}(|\nabla \rho_h|)\nabla \rho_h}^2 \le \int_\Omega H(x,t).
\eeq
With $\delta=1$,  combining \eqref{a} and \eqref{b} leads to \eqref{limsupGrad}.

Assume \eqref{small:atinf2} then \eqref{a} and \eqref{b} lead to 
\beqs
 \begin{split}
 \limsup_{t\to \infty}\norm{\nabla \rho_h}_{0,\beta}^\beta &\le C\left( \frac {\delta}{1+\delta} \right)^{-a}  \left[\limsup_{t\to\infty}\left(\norm{\phi}_{1}^2+\norm{\phi_t}_{1}^2 + \norm{\ft}^2\right)+\limsup_{t\to\infty} \norm{ \rhot_h}^2 \right]+C\delta^\beta\\
 &= C\delta^\beta\to 0 \quad \text { as } \delta\to 0.
 \end{split}
 \eeqs
 We complete the proof.   
\end{proof}
%====================================

The result  \eqref{res2} in Theorem \ref{Est4Grad} provides an estimate for the gradient of density at the given time $t=T$, which includes information on boundary data for all time $t \le T.$  When $T$ is large, it needs to be expressed mainly in terms of the boundary data on the interval $[T-1,T],$ uniformly for all $T$.
%=============================%
\begin{theorem}\label{UniGrad}
Assume $\tilde \rho_h$ a solution to the problem \eqref{semidiscreteform}.  The following inequalities hold uniformly 
\begin{enumerate} 
\item[i.]  For all $t\ge 1$, 
\beq\label{gw2}
\begin{split}
\int_{t-1}^{t} \norm{\nabla \rho_h(\tau)}_{0,\beta}^{\beta}d\tau \le C\left (1+\norm{ \rhot_h(t-1)}^2 +\int_{t-1}^{t}\left(\norm{\phi(\tau)}_{1}^2 +\norm{\ft(\tau)}^\lambda \right) d\tau\right);
\end{split}
\eeq
\item[ii.] For all $t\ge 1$,
\begin{multline} \label{wteq8}
\int_{t-\frac 12}^t\norm{\rhot_{h,t}(\tau)}^2 d\tau+  \norm{\nabla \rho_h(t)}_{0,\beta}^{\beta} \le C\norm{ \rhot_h(t-1)}^2\\
+C \left(1 +\int_{t-1}^{t} \left(\norm{\phi_t(\tau)}_{1}^2+ \norm{\ft(\tau)}^\lambda\right)  d\tau \right);
\end{multline}
\item[iii.] For all $t\ge 1$,
\beq\label{wteq9}
\begin{aligned}
& \norm{\nabla \rho_h(t)}_{0,\beta}^{\beta}\le C\left(1+\norm{\rhot^0}^2+ (Env\, g(t))^\frac{2}{\beta}+\int_{t-1}^{t}\norm{\phi_t(\tau)}_{1}^2d\tau\right).
\end{aligned}
\eeq
\end{enumerate}
\end{theorem}
%=============================%
\begin{proof}
Integrating \eqref{diff-neq} in time from $t-1$ to $t$,  we obtain  
\beq\label{g1w}
\norm{ \rhot_h(t)}^2+\int_{t-1}^{t}\norm{K^{\frac{1}{2}}(|\nabla \rho_h|)\nabla \rho_h}^2d\tau \le \norm{ \rhot_h(t-1)}^2+C\int_{t-1}^{t}\left(\norm{\nabla \pi \phi}^2 +\norm{\ft}+\norm{\ft}^\lambda \right) d\tau.
\eeq
Neglecting nonnegative term on the left hand side of \eqref{g1w} shows that  
\beqs
\int_{t-1}^{t}\norm{K^{\frac{1}{2}}(|\nabla \rho_h|)\nabla \rho_h}^2d\tau \le \norm{ \rhot_h(t-1)}^2+C\int_{t-1}^{t}\left(\norm{\nabla \pi \phi}^2 +\norm{\ft}+\norm{\ft}^\lambda \right) d\tau.
\eeqs
Using \eqref{iineq2} and Young's inequality, we obtain \eqref{gw2}.  

Following by applying Cauchy's inequality to \eqref{rhot} that  
\beq
\norm{ \rhot_{h,t}}^2 +\frac 1 2 \ddt \int_\Omega H(x,t) dx  
\le\frac 12 \norm{K^{\frac{1}{2}}(|\nabla \rho_h|)\nabla\rho_h}^2+\frac 12\left(\norm{\nabla \pi\phi_t}^2 +\norm{\ft}^2 +\norm{ \rhot_{h,t}}^2\right).
\eeq 
In virtue of \eqref{i:ineq4}, we find that  
\beq\label{rt1}
\norm{ \rhot_{h,t}}^2 + \ddt \int_\Omega H(x,t) dx  
\le\int_\Omega  H(x,t) dx +\norm{\nabla \pi\phi_t}^2 +\norm{\ft}^2.
\eeq 
Integrating \eqref{rt1} in $\tau$ from $s$ to $t$ where $s\in[t-1,t]$, we have 
\beq\label{wteq5}
\begin{aligned}
\int_s^t \norm{\rhot_{h,t}}^2 d\tau &+ \int_\Omega H(x,t)dx \\
& \le \int_\Omega H(x,s)dx+\int_s^t \int_\Omega H(x,s)dx d\tau + \int_s^t\left( \norm{\nabla \pi \phi_t}^2 +\norm{\ft}^2  \right)  d\tau\\
& \le \int_\Omega H(x,s)dx+\int_{t-1}^t\int_\Omega H(x,t)dxd\tau  +\int_{t-1}^t \left(\norm{\nabla\pi \phi_t}^2+ \norm{\ft}^2\right) d\tau.
\end{aligned}
\eeq
Then integrating \eqref{wteq5} in $s$ from $t-1$ to $t$ gives 
\beq\label{wteq6}
\begin{aligned}
\int_{t-1}^t\int_s^t\norm{\rhot_{h,t} }^2 d\tau ds + \int_\Omega H(x,t)dx \le 2\int_{t-1}^t\int_\Omega H(x,\tau)dxd\tau +\int_{t-1}^t \left(\norm{\nabla\pi\phi_t}^2+ \norm{\ft}^2\right) d\tau.
\end{aligned}
\eeq
We bound the right hand side in \eqref{wteq6}  using \eqref{i:ineq4}, \eqref{gw2} and Young's inequality to obtain    
\beq\label{gw3}
\begin{aligned}
\int_{t-1}^t\int_s^t\norm{\rhot_{h,t} }^2 d\tau ds + \int_\Omega H(x,t)dx  \le 2\norm{ \rhot_h(t-1)}^2+C \Big(1 +\int_{t-1}^{t} \big(\norm{ \phi_t}_1^2+ \norm{\ft}^\lambda\big)  d\tau \Big).
\end{aligned}
\eeq
The first term of \eqref{gw3} is bounded by   
\beq\label{ft0}
\begin{aligned}
\int_{t-1}^t\int_s^t\norm{\rhot_{h,t} }^2 d\tau ds 
 \ge \int_{t-1}^{t-\frac 12}\int_{t-\frac 12}^t\norm{\rhot_{h,t}}^2  d\tau ds\ge  \frac 12\int_{t-\frac 12}^t\norm{\rhot_{h,t}}^2 d\tau.
\end{aligned}
\eeq
Combining \eqref{gw3}, \eqref{ft0} and using \eqref{iineq2} we obtain \eqref{wteq8}.

The inequality  \eqref{wteq9} follows by using \eqref{res1} to bound the first term of the right hand side in \eqref{wteq8}.
\end{proof}

%====================================
Now we prove the time derivative of pressure is bounded. 
\begin{theorem}\label{phderv} Let $0<t_0<1,$ assume $\tilde \rho_h$ solves  the semidiscrete problem \eqref{semidiscreteform}. Then, there exists a positive constant $C$ such that for all $t\ge t_0$,        
\beq\label{mid1}
 \begin{aligned}
   \norm{\rhot_{h,t}(t)}^2&\le C\Bb(t).
   \end{aligned}
   \eeq 
  where 
  \beq\label{Updef}
 \begin{aligned}
   \Bb(t)= t_0^{-1}e^{-\frac 14 (t-t_0)} \left\{1+\norm{\rho^0}_1^2+ \norm{\phi^0}^2+\int_{0}^{t_0} \big(   \norm{\phi_t(s)}_{1}^2+(Env \, g(s))^{\frac 2\beta}\big) ds\right\}\\
  +\int_{0}^t e^{-\frac 14(t-s)} \Big(1+\norm{\rhot^0}^2+ \norm{\phi_t(s)}_{1}^2+ \norm{\ft_t(s)}^2+(Env \, g(s))^{\frac 2\beta}\Big) ds.
   \end{aligned}
   \eeq 
    Furthermore,
 \beq\label{limsup:Rhot}
 \limsup_{t\to\infty} \norm{\rhot_{h,t}(t)}^2 \le C \Big(1+ (\limsup_{t\to\infty}g(t))^{\frac 2 \beta} + \limsup_{t\to\infty}\norm{\phi_t(t)}_{1}^2+ \limsup_{t\to\infty}\norm{\ft_t(t)}^2\Big). 
 \eeq
  \end{theorem}
%=========================
\begin{proof}
Differentiating \eqref{semidiscreteform} with respect $t$ yields that
 \begin{align*}
(\rhot_{h,tt}, w_h) + \intd{K(|\nabla \rho_h|)\nabla \rho_{h,t},\nabla w_h} =- \intd{K'(|\nabla \rho_h|)\frac{\nabla \rho_{h}\cdot \nabla \rho_{h,t}}{|\nabla \rho_h|}\nabla \rho_h ,\nabla w_h } + \intd{\ft_t, w_h}.
  \end{align*}
Then choosing $ w_h= \rhot_{h,t},$ we obtain
    \begin{multline*}
  \frac{1}{2}\ddt \norm{\rhot_{h,t}}^2 +\norm{K^{\frac{1}{2}}(|\nabla \rho_h|)\nabla \rho_{h,t}}^2  =\intd{K(|\nabla \rho_h|)\nabla \rho_{h,t},\nabla \pi\phi_t}\\
   - \intd{K'(|\nabla \rho_h|)\frac{\nabla \rho_{h}\cdot \nabla \rho_{h,t}}{|\nabla \rho_h|}\nabla \rho_h ,\nabla \rho_{h,t}   }+ \intd{K'(|\nabla \rho_h|)\frac{\nabla \rho_{h}\cdot \nabla \rho_{h,t}}{|\nabla \rho_h|}\nabla \rho_h ,\nabla \pi\phi_t   }+ \intd{ \ft_t, \rhot_{h,t} }.
  \end{multline*}
  The Cauchy inequality and the upper boundedness of $K(\cdot)$ give
  \beq\label{I-1}
   \left|\intd{K(|\nabla \rho_h|)\nabla \rho_{h,t},\nabla \pi\phi_t}\right| \le \frac{1-a}{2}\norm{K^{\frac{1}{2}}(|\nabla \rho_h|)\nabla \rho_{h,t}}^2 +C\norm{\nabla \pi\phi_t  }^2.
  \eeq 
Following from \eqref{i:ineq3} that
   \beq\label{I-2}
  \left| - \intd{K'(|\nabla \rho_h|)\frac{\nabla \rho_{h}\cdot \nabla \rho_{h,t}}{|\nabla \rho_h|}\nabla \rho_h ,\nabla \rho_{h,t}   }\right|\le a\norm{K^{\frac{1}{2}}(|\nabla \rho_h|)\nabla \rho_{h,t}}^2.
  \eeq
 Combining \eqref{i:ineq3}, Cauchy's inequality and the upper boundedness of $K(\cdot)$ gives    
  \beq\label{I-3}
  \begin{split}
  \left| \intd{K'(|\nabla \rho_h|)\frac{\nabla \rho_{h}\cdot \nabla \rho_{h,t}}{|\nabla \rho_h|}\nabla \rho_h ,\nabla \pi\phi_t   } \right|&\le a\int_\Omega K(|\nabla \rho_h|)|\nabla \rho_{h,t}| |\nabla \pi\phi_t | dx\\
          &\le \frac{1-a}{2}  \norm{K^{\frac{1}{2}}(|\nabla \rho_h|)\nabla \rho_{h,t}}^2+C\norm{\nabla \pi\phi_t}^2.
  \end{split}
  \eeq
Using Cauchy's inequality,  
\beq\label{I-4}
\left| ( \ft_t,\rhot_{h,t} ) \right|
\le \frac 1 8 \norm{\rhot_{h,t}}^2 + C \norm{\ft_t}^2.
\eeq
  Above estimates lead to  
  \beq\label{midstep0}
\frac{1}{2}\ddt \norm{\rhot_{h,t}}^2  \le C\norm{\nabla \pi\phi_t}^2 + \frac 1 8 \norm{\rhot_{h,t}}^2 + C \norm{\ft_t}^2.
 \eeq
Combining \eqref{midstep0} and \eqref{keyest2}, we obtain 
 \beq
\begin{aligned}
 &\frac 14 \left(\norm{ \rhot_{h,t}}^2  +\int_\Omega H(x,t) dx\right) 
 +\frac 12 \ddt\left( \int_\Omega H(x,t)dx + \norm{\rhot_{h,t}}^2 \right)\\
  &\qquad \le  C\left(\norm{ \nabla \pi\phi }^2+\norm{\ft}^2 + \norm{\rhot_h}^2 + \norm{\nabla \pi\phi_t}^2 +   \norm{\ft_t}^2\right) + \frac 1 8 \norm{\rhot_{h,t}}^2.
 \end{aligned}
 \eeq
   Define $$\mathcal E(t)= \norm{\rhot_{h,t}}^2+  \int_\Omega H(x,t) dx. $$
%\beq
%\begin{aligned}
%&\frac 1 8 \mathcal E(t) +\frac 12 \ddt  \mathcal E(t)\\
%&\quad \le  \norm{ \nabla \pi\phi }^2+2\norm{\ft}^2 +\frac 34 \norm{\rhot_h}^2 + C\norm{\nabla \pi\phi_t}^2 +  2 \norm{\ft_t}^2.
%\end{aligned}
%\eeq
%   
%
%  \begin{align*}
% \ddt \mathcal E(t)  +\frac 14 \mathcal E(t) \le  C(\norm{ \nabla \pi\phi }^2 + \norm{\nabla \pi\phi_t}^2+ \norm{\rhot_h}^2+\norm{\ft}^2 + \norm{\ft_t}^2).  
%  \end{align*}
 Using \eqref{res1}, we obtain    
 \beq\label{bb}
\begin{split}
\ddt \mathcal E(t) +\frac 14 \mathcal E(t) \le C\left( \norm{\rhot_h}^2 + \norm{\phi}_{1}^2+\norm{\phi_t}_{1}^2+\norm{\ft}^2 + \norm{\ft_t}^2\right).
 \end{split} 
  \eeq
For any $t\ge t_0 \ge t'>0$, integrating \eqref{bb} from $t'$ to $t$, we find that  
  \beq
\begin{split}
\mathcal E(t) \le  e^{-\frac 14 (t-t')}\mathcal E(t') +  C\int_{0}^t e^{-\frac 14(t-s)} \left[ \norm{\rhot_h}^2 + \norm{\phi}_{1}^2+\norm{\phi_t}_{1}^2+\norm{\ft}^2 + \norm{\ft_t}^2\right] ds.
  \end{split}
  \eeq
 Integrating in $t'$ from $0$ to $t_0$ gives
\beqs
 \begin{split}
 t_0 \mathcal E(t)  \le  e^{-\frac 14 (t-t_0)} \int_0^{t_0}   \mathcal E(t') dt' +t_0C\int_{0}^t e^{-\frac 14(t-s)} \left[ \norm{\rhot_h}^2 + \norm{\phi}_{1}^2+\norm{\phi_t}_{1}^2+\norm{\ft}^2 + \norm{\ft_t}^2\right] ds.
 \end{split}
  \eeqs
  Integrating \eqref{Grad-tder-eq} from $0$ to $t_0$ we obtain  
\beq\label{E:est}
\int_0^{t_0} \mathcal E(t') dt'  \le \int_\Omega H(x,0) dx+  \int_{0}^{t_0} \Big[ \norm{ \phi}_{1}^2+\norm{\phi_t  }_{1}^2 + \norm{\ft}^2 + \norm{ \rhot_h}^2 \Big] ds.
\eeq
%Using \eqref{E:est} gives   
%\beq
%\int_0^t \mathcal E(\tau) d\tau  \le \int_\Omega H(x,0) dx+  \int_{0}^{t_0} \Big[ \norm{ \phi}_{1}^2+\norm{\phi_t  }_{1}^2 + \norm{\ft}^2 + \norm{ \rhot_h}^2 \Big] ds.
%\eeq  
Therefore,
\beq\label{al1}
 \begin{aligned}
  \mathcal E(t) &\le t_0^{-1}e^{-\frac 14 (t-t_0)} \Big\{\int_\Omega H(x,0) dx+  \int_{0}^{t_0} \Big[ \norm{ \phi}_{1}^2+\norm{\phi_t  }_{1}^2 + \norm{\ft}^2 + \norm{ \rhot_h}^2 \Big] ds\Big\}\\
  &+C\int_{0}^t e^{-\frac 14(t-s)} \left[\norm{\rhot_h}^2 + \norm{\phi}_{1}^2+\norm{\phi_t}_{1}^2+\norm{\ft}^2 + \norm{\ft_t}^2\right] ds.
   \end{aligned}
   \eeq
Then the inequality \eqref{mid1} follows from \eqref{al1},\eqref{res1}, \eqref{iineq2} and \eqref{i:ineq4}. 

Now applying Gronwall's inequality in \eqref{ulode} for \eqref{bb} yields   
\beqs
  \begin{split}
 \limsup_{t\to\infty} \mathcal E(t) &\le C \limsup_{t\to\infty}\norm{\rhot_h}^2 + C\limsup_{t\to\infty}\left(\norm{\phi}_{1}^2+\norm{\phi_t}_{1}^2+\norm{\ft}^2 +\norm{\ft_t}^2\right)\\
 &\le C \left( (\limsup_{t\to\infty}g(t))^{\frac 2 \beta} + \limsup_{t\to\infty}\norm{\phi_t}_{1}^2+ \limsup_{t\to\infty}\norm{\ft_t}^2+1\right). 
 \end{split}
\eeqs
The last inequality  is followed by using \eqref{limsup:rhot}. Then the proof is complete.  
\end{proof}

Next, we derive a uniformly bound for $\rhot_{h,t}$.
%=======================================%
\begin{theorem}\label{Uni-rhot} Suppose $\tilde \rho_h$ solves  the semidiscrete problem \eqref{semidiscreteform}. Then, there exists a positive constant $C$ such that for all $t\ge 1$,  
\beq\label{q34}
\norm{\rhot_{h,t}(t) }^2
\le C\left\{1+\norm{\rhot^0}^2+(Env\ \ g(t))^\frac{2}{\beta} +\int_{t-1}^{t} \left(\norm{\phi_t(s)}_{1}^2+  \norm{\ft_t(s)}^2 \right)  ds \right\}.
\eeq
\end{theorem}
%=======================================%
\begin{proof}
Integrating \eqref{midstep0} in time from $s$ to $t$ where $t-\frac 12\le s\le t$, we have 
\beq\label{q31}
\begin{aligned}
\norm{\rhot_{h,t}(t) }^2 
&\le \norm{\rhot_{h,t}(s) }^2 +\frac 14  \int_s^t\norm{\rhot_{h,t}(\tau) }^2 d\tau +C\int_s^t\big(\norm{\nabla \pi \phi_t }^2+\norm{\ft_t}^2 \big)d\tau\\
&\le \norm{\rhot_{h,t}(s) }^2 +\frac 14  \int_{t-\frac 12}^t\norm{\rhot_{h,t}(\tau) }^2 d\tau +C\int_{t-1}^t\big(\norm{\nabla \pi \phi_t }^2+\norm{\ft_t}^2 \big)d\tau.
\end{aligned}
\eeq
Now integrating \eqref{q31} in $s$ from $t-\frac 12$ to $t$, we have 
\beq%\label{q32}
\begin{aligned}
\norm{\rhot_{h,t}(t) }^2
&\le \frac 5 2 \int_{t-\frac 12}^t\norm{\rhot_{h,t}(s) }^2 ds +C\int_{t-1}^t\big(\norm{\nabla \pi \phi_t }^2+\norm{\ft_t}^2 \big) d\tau.
\end{aligned}
\eeq
This and \eqref{wteq8} yield   
\beq\label{q33}
\norm{\rhot_{h,t}(t) }^2 
\le C \norm{ \rhot_h(t-1)}^2+C \Big(1 +\int_{t-1}^{t} \left(\norm{\phi_t}_{1}^2+ \norm{\ft}^\lambda + \norm{\ft_t}^2 \right)  d\tau \Big).
\eeq
We use \eqref{res1} to estimate the first term on the right hand side of \eqref{q33}. Then  \eqref{q34} follows.
\end{proof}

%=================================
The proof of the following stability results applied to the problem \eqref{weakform} is similar to the proofs in Theorem~\ref{bound-lq}--\ref{Uni-rhot}. We  omit for brevity. 
\begin{theorem} \label{B4rho}
Let $\tilde \rho$ be a solution of the problem \eqref{weakform}. Then, there exists a positive constant $C$ such that 
\begin{itemize}
\item[i.] For all $t>0,$
 \begin{align}
\label{rhoEst}
&\norm{\rhot (t)}\le  C\left(1+\norm{\rho^0 } + ( Env \, g(t)) ^{\frac 1 \beta}\right);\\
\label{limsup-rho}
&\limsup_{t\to\infty}\norm{\rhot(t)}^2\le  C \left(1+ \limsup_{t\to\infty}g(t)  \right)^{\frac 2 \beta}.\\
& \text{ If }   \limsup_{t\to \infty} \norm{\ft(t)} = \limsup_{t\to \infty}\norm{\phi(t)}_{1}=0  \text{  then }
\label{res1ab} 
\limsup_{t\to\infty} \norm{\rhot (t)}^2=0.
\end{align}
\item[ii.] For all $t>0,$
\begin{align}
\label{res20}
&\norm{\nabla \rho(t)}_{0,\beta}^{\beta}  \le C \Aa(t), \text { where } \Aa(t) \text { is defined in \eqref{Mdef}.}  \\
\label{limsupGrad0}
&\limsup_{t\to \infty} \norm{\nabla \rho(t)}_{0,\beta}^\beta \le C\left(1+\limsup_{t\to\infty}\norm{\phi_t(t) }_{1}^2+(\limsup_{t\to\infty} \, g(t))^{\frac{2}{\beta}}\right).\\
& \text {If } \limsup_{t\to \infty} \left\{ \norm{\ft(t), \norm{\phi(t)}_{1}, \norm{\phi_t(t)}_{1}}\right\} = 0
\text { then }
 \label{Grad:small0}
 \limsup_{t\to \infty}\norm{\nabla \rho(t)}_{0,\beta}  =0.
 \end{align}
\item[iii.] Let $0<t_0<1$, for all $t\ge t_0$,        
 \begin{align}
\label{mid10} 
&\norm{\rhot_{t}(t)}^2\le C\Bb(t),\text { where } \Bb(t) \text{ is defined in \eqref{Updef}. } \\
\label{limsup:Rhot0}
&\limsup_{t\to\infty} \norm{\rhot_{t}(t)}^2 \le C \Big( 1+(\limsup_{t\to\infty}g(t))^{\frac 2 \beta} + \limsup_{t\to\infty}\norm{\phi_t(t)}_{1}^2+ \limsup_{t\to\infty}\norm{\ft_t(t)}^2\Big),   
\end{align}
\item[iv.] For all $t\ge 1,$  
\begin{align}
 &\norm{\nabla \rho(t)}_{0,\beta}^{\beta}\le C\left(1+\norm{\rhot^0}^2+ (Env \, g(t))^\frac{2}{\beta}+\int_{t-1}^{t}\norm{\phi_t(\tau)}_{1}^2d\tau\right).\\
 \label{q340} &\norm{\rhot_{t}(t) }^2
\le C\left\{1+\norm{\rhot^0}^2+(Env\,  g(t))^\frac{2}{\beta} +\int_{t-1}^{t} \left(\norm{\phi_t(s)}_{1}^2+  \norm{\ft_t(s)}^2 \right)  ds \right\}.
\end{align}
\end{itemize}
\end{theorem}

%==================================================================%
\section{Error estimates}\label{errSec} In this section, we will establish the error estimates between analytical solution and approximation solution in several norms. Let 
\beq\label{Ndef}
\Ff(t)=
\begin{cases}
\displaystyle  1+\norm{ \rho^0}_1^2+ \norm{\phi^0}^2 +\int_0^t e^{-\frac 1 2 (t-s)} \left(  \norm{\phi_t(s)  }_{1}^2 +  (Env \, g(s))^{\frac 2 {\beta}}\right),  & \text { if } 0\le t\le 1,\\
\displaystyle1+\norm{\rhot^0}^2+ (Env\, g(t))^\frac{2}{\beta}+\int_{t-1}^{t}\norm{\phi_t(\tau)}_{1}^2d\tau, &\text{ if } t\ge 1, 
\end{cases}
\eeq
and
\beq\label{L1L2}
\mathcal K = 1+\limsup_{t\to\infty}\norm{\phi_t(t) }_{1}^2
+(\limsup_{t\to\infty} g(t))^{\frac 2 \beta}, \quad  \mathcal L = \mathcal K + \limsup_{t\to\infty}\norm{\ft_t(t)}^2.
\eeq 
 Define 
  \beq\label{lambda}
  \Lambda(t) = 1+ \norm{\nabla \rho(t)}^\beta_{0,\beta}+\norm{\nabla \rho_h(t)}^\beta_{0,\beta}. 
  \eeq
  
 According to Theorem \ref{Est4Grad}, ~\ref{UniGrad} and Theorem~\ref{B4rho}
  \beqs
  \Lambda(t) \le C\Ff(t), \quad \text{and}\quad \limsup_{t\to\infty} \Lambda (t)\le \mathcal K.
  \eeqs
  Let 
  \beq\label{Pdef}
\Gg(t)=
\begin{cases}
\displaystyle  t_0^{-1} \Big\{1+\norm{\rho^0}_1^2+ \norm{\phi^0}^2+\int_{0}^{t_0} 
\big(   \norm{\phi_t(s)}_{1}^2+(Env \, g(s))^{\frac 2\beta}\big) ds\Big\}\\
\displaystyle  +\int_{0}^t e^{-\frac 14(t-s)} \Big(1+\norm{\rhot^0}^2+ \norm{\phi_t(s)}_{1}^2+ \norm{\ft_t(s)}^2+(Env\, g(s))^{\frac 2\beta}\Big) ds, & \text { if } 0<t_0\le t\le 1,\\
\displaystyle  1+\norm{\rhot^0}^2+(Env\, g(t))^\frac{2}{\beta} +\int_{t-1}^{t} \left(\norm{\phi_t(s)}_{1}^2+  \norm{\ft_t(s)}^2 \right), &\text{ if } t\ge 1. 
\end{cases}
\eeq
We have from Theorem~\ref{phderv},~\ref{Uni-rhot} and Theorem~\ref{B4rho} that  
\beq\label{dert-rho}
\norm{\rhot_t(t)}+\norm{\rhot_{h,t}(t)}\le C\sqrt{\Gg(t)} \quad \text { and } \quad \limsup_{t\to\infty}\left(\norm{\rhot_t(t)}+\norm{\rhot_{h,t}(t)} \right)\le C\mathcal L. 
\eeq
\subsection{Error estimate for continuous Galerkin method} We will find the error bounds in the semidiscrete method by comparing the computed solution to the projections of the exact solutions. To do this, we restrict the test functions in \eqref{weakform} to the finite dimensional spaces. Let
\beq
\chi=\rhot -\rhot_h = (\rhot -\pi\rhot) - (\rhot_h- \pi\rhot) \eqdef \vartheta -\theta_h,\quad \text { and } \quad  \varphi\eqdef \phi-\pi\phi. 
\eeq

%==============================
% L^2 errror estimate 
%==============================

\begin{theorem}\label{longerr1} Let $1\le k\le r+1$,  $\tilde \rho,\tilde \rho_h$ be solutions to \eqref{weakform} and \eqref{semidiscreteform} respectively. Assume that $\rhot\in L^\infty(\R_+,H^k(\Omega))$, $\rhot_t\in L^2(\R_+,H^k(\Omega) )$.  Then, there exists a constant positive constant $C$ independent of $h$ such that for all $t> 0,$    
 \beq\label{err-rho}
 \norm{(\rhot -\rhot_h)(t)}^2 \le C h^{2k}\norm{\rhot(t)}_{k}^2+ C h^{k-1}\int_0^t  e^{-2^{-a}\int_s^t \Lambda(\tau)^{-1}d\tau }\Ff (s)\Hh(s) ds,
  \eeq
   where $\Ff(t)$ is defined as \eqref{Ndef}, and 
   \beq\label{Rdef}
   \Hh(t) = \norm{\rhot_t(t)}_{k}^2 
  +\norm{\rhot(t)}_{k,\beta}+ \norm{\rhot(t)}_{k}^2 +\norm{\phi(t)}_{k,\beta}^2+ \norm{\phi(t)}_{k,\beta}.
   \eeq
Furthermore, if $\displaystyle \int_0^\infty \Lambda^{-1}(t)dt =\infty$ then
\beq\label{limsupErr}
 \begin{split}
\limsup_{t\to\infty}\norm{(\rhot-\rhot_h)(t)}^2&\le C h^{k-1}\mathcal K^{2}  \limsup_{t\to\infty} \Hh(t).
  \end{split}
 \eeq
\end{theorem}

\begin{proof}
From \eqref{weakform} and \eqref{semidiscreteform}, we find the error equation 
\beq\label{errEq0}
 (\rhot_{t}- \rhot_{h,t}, w_h) + (K(|\nabla \rho|)\nabla \rho -K(|\nabla \rho_h|)\nabla \rho_h ,\nabla w_h) =0,  \quad \forall w_h\in W_h.
\eeq
Taking $w_h=\theta_h$, we obtain 
\beq\label{errEq1}
 (\vartheta_t-\theta_{h,t} , \theta_h) + (K(|\nabla \rho|)\nabla \rho -K(|\nabla \rho_h|)\nabla \rho_h ,\nabla \theta_h) =0 .
\eeq
We rewrite the equation \eqref{errEq1} as  form  
\beqs
\begin{aligned}
 \frac 1 2\ddt \norm{\theta_h}^2 &+(K(|\nabla \rho|)\nabla \rho -K(|\nabla \rho_h|)\nabla \rho_h ,\nabla(\rho -\rho_h) ) \\
 &=  (\vartheta_t,\theta_h) +  (K(|\nabla \rho|)\nabla \rho -K(|\nabla \rho_h|)\nabla \rho_h ,\nabla \vartheta + \nabla \varphi ). 
 \end{aligned} 
 \eeqs
 Thanks to \eqref{Mono}, 
 \beq \label{monotone}
 \left(K(|\nabla \rho|)\nabla \rho -K(|\nabla \rho_h|)\nabla \rho_h ,\nabla(\rho -\rho_h)\right)  \ge c_4\omega \norm{\nabla(\rho-\rho_h)}_{0,\beta}^2. 
    \eeq
 Using \eqref{Lips}, H\"older's and Young's inequality, we have 
\beq\label{r20}
\begin{aligned}
(K(|\nabla \rho|)\nabla \rho -K(|\nabla \rho_h|)\nabla \rho_h ,\nabla \vartheta + \nabla \varphi )
&\le C(|\nabla \rho|^{\beta-1}+|\nabla \rho_h|^{\beta-1} ,|\nabla \vartheta + \nabla \varphi |)\\
& \le C\left(\norm{\nabla \rho}_{0,\beta}^{\beta-1}+\norm{\nabla \rho_h}_{0,\beta}^{\beta-1}\right)\norm{\nabla \vartheta+ \nabla \varphi }_{0,\beta} \\
&\le C \Lambda(t) \left( \norm{\nabla \vartheta}_{0,\beta}+\norm{ \nabla \varphi }_{0,\beta}\right).
\end{aligned}
\eeq
 
Using Young's inequality, for $\varep>0$  
\beq\label{Yb}
(\vartheta_t, \theta_h) \le C\omega^{-1}\varep^{-1} \norm{\vartheta_t}^2 +\varep\omega \norm{\theta_h}^2.
\eeq
Combining \eqref{monotone},\eqref{r20} and \eqref{Yb} gives 
\beqs
\begin{split}
 \frac 1 2\ddt \norm{\theta_h}^2 +c_4 \omega\norm{\nabla(\rho-\rho_h)}_{0,\beta}^2  
\le  C\omega^{-1}\varep^{-1} \norm{\vartheta_t}^2 +\varep\omega \norm{\theta_h}^2 +   \Lambda(t)\left(\norm{\nabla \vartheta}_{0,\beta}+ \norm{\nabla \varphi}_{0,\beta}\right). 
 \end{split}
 \eeqs
 By Poincar\'e-Sobolev inequality $\norm{u}\le C_p\norm{\nabla u}_{0,\beta}$ for all $u\in H_0^1(\Omega),$
 \beq\label{InvIneq0}
 \begin{aligned}
 \norm{\nabla(\rho-\rho_h)}_{0,\beta}^2 &\ge \frac 1 2 \norm{\nabla(\rhot-\rhot_h)}_{0,\beta}^2 - \norm{\nabla\varphi}_{0,\beta}^2\\
 %&\ge C_1\norm{\rhot-\rhot_h}^2 - \norm{\nabla\varphi}_{0,\beta}^2\\
 &\ge \frac{1}{2C_p^2} \norm{\rhot-\rhot_h}^2 - \norm{\nabla\varphi}_{0,\beta}^2\\
 &\ge \frac{1}{4C_p^2}\norm{\theta_h}^2 - \frac{1}{2C_p^2}\norm{\vartheta}^2 - \norm{\nabla\varphi}_{0,\beta}^2. 
 \end{aligned}
 \eeq
 Here we have used the inequality $(a-b)^2 \ge \frac 12 a^2 - b^2$.  Thus,
 \beqs
\begin{split}
 \frac 1 2\ddt \norm{\theta_h}^2 +\frac{c_4}{4C_p^2} \omega\norm{\theta_h}^2  
&\le  C\omega^{-1}\varep^{-1} \norm{\vartheta_t}^2 +\varep\omega \norm{\theta_h}^2 +   C \omega\left(\norm{\vartheta}^2+  \norm{\nabla \varphi}_{0,\beta}^2\right)\\
&\quad +  C \Lambda(t) \left( \norm{\nabla \vartheta}_{0,\beta}+\norm{ \nabla \varphi }_{0,\beta}\right). 
 \end{split}
 \eeqs
 Taking $\varep =\frac{c_4}{8C_p^2}$, we find that 
   \beq\label{Dineq0}
 \ddt \norm{\theta_h}^2 + \omega\norm{\theta_h}^2  
\le  C\omega^{-1}\norm{\vartheta_t}^2 + C \omega\Big( \norm{\vartheta}^2+ \norm{\nabla\varphi}_{0,\beta}^2\Big) +C \Lambda(t)  \left( \norm{\nabla \vartheta}_{0,\beta}+\norm{ \nabla \varphi }_{0,\beta}\right).
 \eeq
 Observing from \eqref{res2} that  $\omega(t)\le 1$. Following from \eqref{res2} and \eqref{res20}, 
 \beq\label{Odef0}
 \begin{split}
 \omega^{-1}(t)\le \left(1+\norm{\nabla \rho}_{0,\beta}+ \norm{\nabla \rho_h}_{0,\beta}\right)^{\beta\gamma} \le (2^\beta \Lambda(t))^\gamma\le 2^a\Lambda(t)\le  C\Ff(t)  . 
 \end{split}
 \eeq
 Thus,
\beq\label{DIneq}
 \ddt \norm{\theta_h}^2 + 2^{-a}\Lambda(t)^{-1}\norm{\theta_h}^2  
\le  C\Ff(t)\left(\norm{\vartheta_t}^2 +  \norm{\vartheta}^2+ \norm{\nabla\varphi}_{0,\beta}^2 + \norm{\nabla \vartheta}_{0,\beta}+\norm{ \nabla \varphi }_{0,\beta}\right).
 \eeq
Applying Gronwall's inequality and using the fact that $\theta_h(0) =0$, we obtain
% \beq\label{thetaEst}
% \begin{aligned}
%&  \norm{\theta_h}^2 \le \norm{\theta_h(0)}^2 e^{-2^{-a}\int_0^t \Lambda(\tau)^{-1}d\tau }\\
%  &\quad+C\int_0^t e^{-2^{-a}\int_s^t \Lambda(\tau)^{-1}d\tau }\Ff(s) \left(\norm{\vartheta_t}^2 +  \norm{\vartheta}^2+ \norm{\nabla\varphi}_{0,\beta}^2+ \norm{ \nabla \varphi }_{0,\beta} + \norm{\nabla \vartheta}_{0,\beta}\right)ds.
% \end{aligned}
% \eeq
% Since $\theta_h(0) =0$ we simplify above inequality to have
\beq\label{Btheta}
 \norm{\theta_h}^2 \le \int_0^t e^{-2^{-a}\int_s^t \Lambda(\tau)^{-1}d\tau } \Ff(s) \left(\norm{\vartheta_t}^2 +  \norm{\vartheta}^2+ \norm{\nabla\varphi}_{0,\beta}^2 +\norm{ \nabla \varphi }_{0,\beta}+ \norm{\nabla \vartheta}_{0,\beta}\right)ds.
 \eeq
Consequently, 
\beq\label{TIneq}
 \norm{\theta_h}^2 
  \le C h^{k-1}\int_0^t \Ff(s) e^{-2^{-a}\int_s^t \Lambda(\tau)^{-1}d\tau }\left[  \norm{\rhot_t}_{k}^2+\norm{\rhot}_{k}^2+\norm{\phi}_{k,\beta}^2+\norm{\phi}_{k,\beta}
  +\norm{\rhot}_{k,\beta}\right] ds.
  \eeq
 The inequality \eqref{err-rho} follows by the triangle inequality and \eqref{TIneq}. 
%\beq
%\begin{split}
%& \norm{\rhot -\rhot_h}^2 \le C h^{2k}\norm{\rhot}_{k}^2 \\
% &\quad + C h^{k-1}\int_0^t \Ff(s) e^{-2^{-a}\int_s^t \Lambda(\tau)^{-1}d\tau }\left[ \norm{\rhot_t}_{k}^2 +\norm{\rhot}_{k}^2 +\norm{\phi}_{k,\beta}^2+ \norm{\phi}_{k,\beta} +\norm{\rhot}_{k,\beta}\right] ds.
%  \end{split}
% \eeq
%  
   
  Applying Lemma~\ref{ODE2} for \eqref{Dineq0}, we obtain  
 \beqs
 \begin{split}
 \limsup_{t\to\infty}\norm{\theta_h}^2 \le  C\limsup_{t\to\infty}\Big[ \omega^{-2} \norm{\vartheta_t}^2 + \omega^{-1}\Lambda(t)\left(\norm{\nabla \vartheta}_{0,\beta}+\norm{\nabla\varphi}_{0,\beta} \right)+ \norm{\vartheta}^2+ \norm{\nabla\varphi}_{0,\beta}^2\Big].
 \end{split}
 \eeqs
 Hence,
 \beqs
 \begin{split}
 \limsup_{t\to\infty}\norm{\theta_h}^2 &\le  Ch^{k-1}\limsup_{t\to\infty}\Big[ 
\Lambda(t)^{2} \left(\norm{\rhot_t}_{k}^2 + \norm{\rhot}_{k,\beta}+\norm{\phi}_{k,\beta} \right)+ \norm{\rhot}_{k}^2+ \norm{\phi}_{k,\beta}^2\Big]\\
  &\le  Ch^{k-1}\left(\limsup_{t\to\infty}\Lambda(t) \right)^{2} 
 \limsup_{t\to\infty}\Hh(t)\\
 & \le  Ch^{k-1}\mathcal K^{2}\limsup_{t\to\infty}\Hh(t),
  \end{split}
 \eeqs
which shows \eqref{limsupErr}. The proof is complete.  
 \end{proof}

%==============================

  The $L^2$-error estimate and the inverse estimate enable us to have the $L^\infty$-error estimate as the following 
\begin{corollary}  Let $1\le k\le r+1,$ $\tilde\rho,\tilde\rho_h$ be solutions to \eqref{weakform} and \eqref{semidiscreteform} respectively.  Assume that $\rhot\in L^\infty(\R_+,W^{k,\infty}(\Omega) )$,  $\rhot_t\in L^2(\R_+,H^k(\Omega) )$.  Then, there exists a constant positive constant $C$ independent of $h$ such that    
\beq\label{ErrRhoLinf}
 \norm{(\rhot -\rhot_h)(t)}_{0,\infty}^2 \le C h^{2k}\norm{\rhot(t)}_{k,\infty}^2+ C h^{k-1-d} \int_0^t  e^{-2^{-a}\int_s^t \Lambda(\tau)^{-1}d\tau }\Ff (s)\Hh(s) ds,
  \eeq
   where $\Ff(t)$, $\Hh(t)$ are defined in \eqref{Ndef} and \eqref{Rdef} respectively.
   \end{corollary}
\begin{proof} 
Recall that in the quasi-uniform of $\mathcal T_h$ we have the inverse estimate (see in \cite{BS08,TV06})
\beqs
\norm{\theta_h}_{0,\infty}\le Ch^{-\frac d  2}\norm{\theta_h}.
\eeqs
This and triangle inequality imply that 
\beqs 
   \norm{\chi}_{0,\infty}^2 
  \le   2\norm{\vartheta}_{0,\infty}^2+2\norm{\theta_h}_{0,\infty}^2 \le  Ch^{2k}\norm{\rhot}_{k,\infty}^2+Ch^{-d}\norm{\theta_h}^2.
 \eeqs 
 The inequality \eqref{ErrRhoLinf} follows directly from\eqref{TIneq}.   
\end{proof}

Now we give an error estimate for the gradient vector.
\begin{theorem}\label{longerr2} Let $1\le k\le r+1$. Assume that $\rhot\in L^\infty(\R_+,H^k(\Omega))$, $\rhot_t\in L^2(\R_+,H^k(\Omega) )$. Let $\tilde\rho,\tilde\rho_h$ be solutions to \eqref{weakform} and \eqref{semidiscreteform} respectively.  There exists a positive constants $C$ independent of $h$ such that for any $t\ge t_0>0,$    
\beq\label{ErrGrad}
\begin{aligned}
 \norm {\nabla (\rho -\rho_h)(t)}_{0,\beta}^2  \le Ch^{\frac{k-1}2} \Ff^2(t) \left(\Big(\Gg(t)\int_0^t e^{-2^{-a}\int_s^t \Lambda(\tau)^{-1}d\tau } \Ff(s)\Hh(s) ds\Big)^{\frac 12} + \Hh(t)\right), 
\end{aligned}
\eeq
where 
%\beqs
%\mathcal H(t) =  
% \left\{\Gg(t)\int_0^t e^{-2^{-a}\int_s^t \Lambda(\tau)^{-1}d\tau } \Ff(s)\Hh(s) ds\right\}^{\frac 12} + \Hh(t),
%\eeqs
%and 
$\Ff(t), \Gg(t), \Hh(t)$ are defined in \eqref{Mdef}, \eqref{Pdef}, \eqref{Rdef} respectively.\\
 Furthermore, if $\displaystyle \int_0^\infty \Lambda^{-1}(t)dt =\infty$ then
\beq\label{limsupErrGrad}
\limsup_{t\to\infty} \norm {\nabla (\rho -\rho_h)(t)}_{0,\beta}^2\le  C h^{\frac{k-1}2}\mathcal K^3\mathcal L  \left( \limsup_{t\to\infty} \Hh(t)+ 
\Big( \limsup_{t\to\infty} \Hh(t) \Big)^{\frac 12}
\right).
\eeq 
\end{theorem}
\begin{proof}
We rewrite equation \eqref{errEq1} as the following    
\begin{multline}\label{sq1}
 \intd{K(|\nabla \rho|)\nabla \rho -K(|\nabla \rho_h|)\nabla \rho_h ,\nabla \rho -\nabla\rho_h }\\
  =(\rho_t -\rho_{h,t}, \theta_h) + \intd{K(|\nabla \rho|)\nabla \rho -K(|\nabla \rho_h|)\nabla \rho_h ,\nabla \vartheta +\nabla \varphi }.
\end{multline}
According to  \eqref{Mono},
\beq\label{posOper}
(K(|\nabla \rho|)\nabla \rho -K(|\nabla \rho_h|)\nabla \rho_h ,\nabla \rho -\nabla\rho_h )\ge c_4\mathcal \omega \norm{\nabla (\rho-\rho_h)}_{0,\beta}^2.
\eeq
Using  H\"older's inequality and \eqref{r20}, we find that
\begin{multline}\label{sq3}
(\rho_t -\rho_{h,t}, \theta_h) + \intd{K(|\nabla \rho|)\nabla \rho -K(|\nabla \rho_h|)\nabla \rho_h ,\nabla \vartheta +\nabla \varphi}\\
 \le  C(|\rho_t|+|\rho_{h,t}|,|\theta_h|) + C(|\nabla \rho|^{\beta-1}+ |\nabla \rho_h|^{\beta-1},|\nabla \vartheta|+|\nabla \varphi|)\\
 \le C\left(\norm{\rho_t}+\norm{\rho_{h,t}} \right)\norm{\theta_h} + C\Lambda(t)(\norm{\nabla \vartheta}_{0,\beta}+\norm{\nabla\varphi}_{0,\beta}).
\end{multline}
Combining \eqref{sq1} -- \eqref{sq3}, and \eqref{Odef0} yields  
\beq\label{errgrad}
\begin{split}
 \norm {\nabla (\rho -\rho_h)}_{0,\beta}^2&\le C\omega^{-1}\left(\norm{\rho_t}+\norm{\rho_{h,t}} \right)\norm{\theta_h} + C\omega^{-1}\Lambda(t)(\norm{\nabla \vartheta}_{0,\beta}+\norm{\nabla\varphi}_{0,\beta})\\
 &\le C\Lambda(t)^2\Big[ \left(\norm{\rho_t}+\norm{\rho_{h,t}} \right)\norm{\theta_h} +\norm{\nabla \vartheta}_{0,\beta} +\norm{\nabla\varphi}_{0,\beta}\Big]\\
&\le C\Ff^2(t)\Big[ \left(\norm{\rho_t}+\norm{\rho_{h,t}} \right)\norm{\theta_h} +\norm{\nabla \vartheta}_{0,\beta}+\norm{\nabla\varphi}_{0,\beta} \Big].
\end{split}
\eeq

%Thanks to \eqref{dert-rho}
%\beq
%\norm{\rho_t}+\norm{\rho_{h,t}} \le C\sqrt{\Gg(t)}.
%\eeq

Due to \eqref{dert-rho}, \eqref{TIneq} and the fact that $\norm{\nabla \vartheta}_{0,\beta}\le Ch^{k-1}\norm{\rhot}_{k,\beta}$, the left hand side of \eqref{errgrad} is bounded by  
\beq\label{lhs0}
\begin{split}
C\Ff^2(t) \left\{ h^{k-1}\Gg(t)\int_0^t e^{-2^{-a}\int_s^t \Lambda(\tau)^{-1}d\tau } \Ff(s)\Hh(s) ds\right\}^{\frac 12}\\
 + Ch^{k-1}\Ff^2(t)\left(\norm{\rhot}_{k,\beta}+\norm{\phi}_{k,\beta}\right).
\end{split}
\eeq
The inequality \eqref{ErrGrad} follows from \eqref{errgrad} and \eqref{lhs0}.

Take limit superior both sides of \eqref{errgrad}, we find that 
\beqs
\begin{split}
& \limsup_{t\to\infty}\norm {\nabla (\rho -\rho_h)}_{0,\beta}^2\\
&\quad \le C\limsup_{t\to\infty}\Lambda(t)^2\Big[ \limsup_{t\to\infty}\Big(\norm{\rho_t}+\norm{\rho_{h,t}} \Big)\limsup_{t\to\infty}\norm{\theta_h}
  +\limsup_{t\to\infty}\left(\norm{\nabla \vartheta}_{0,\beta}+\norm{\nabla \varphi}_{0,\beta}\right) \Big]\\
&\quad \le C h^{\frac{k-1}2}\mathcal K^2\Big\{ \mathcal K \mathcal L\Big( \limsup_{t\to\infty} \Hh(t)\Big)^{\frac 12} +  \limsup_{t\to\infty} \left(\norm{\rhot}_{k,\beta} +  \norm{\phi}_{k,\beta}\right) \Big\} .
\end{split}
\eeqs
Therefore, we obtain \eqref{limsupErrGrad}.
 \end{proof}

%================================
\subsection{Error analysis for fully discrete Galerkin method}
%=======================================
 In analyzing this method, proceed in a similar fashion as for the semidiscrete method. We derive an error estimate for the fully discrete time Galerkin approximation of the differential equation. First, we give some uniform stability results  that are crucial in
getting the  convergence results.
%\subsection{ Estimate for solution} 

\begin{lemma} Let  $\rhot_h^n$ solve the fully discrete Galerkin finite element
approximation \eqref{fullydiscreteform} for each time step $n\ge 1 $. Then, there exists a positive constant $C$ independent of $t,n, \Delta t$ satisfying
\beq\label{pwBound}
\norm{\rhot_h^n} \le   C\max\left\{ \norm{\rhot^0}, 1+ \norm{\ft^n}^{\frac{1}{\beta-1}} + \norm{\phi^n}_{1}^{\frac{2}{\beta}} \right \} . 
\eeq
For all $i=1\ldots n,$
\beqs
\begin{split}
\sum_{j=i}^n \Delta t\norm{ \nabla \rhot_h^j}_{0,\beta}^\beta&\le C\max \left\{ \norm{\rhot^0}^2, 1+ \norm{\ft^{i-1}}^{\frac{2}{\beta-1}} +\norm{\phi^{i-1}}_{1}^\frac{4}{\beta} \right \}\\
&+ C\sum_{j=i}^m\Delta t \left(1+ \norm{\ft^j}^{\lambda} +\norm{\phi^j}_{1}^2\right).
\end{split}
\eeqs
\end{lemma}
\begin{proof}
Choosing $w_h =2\Delta t\rhot_h^n$ in \eqref{fullydiscreteform} and  using identity
 \beqs
 2\Big(  \rhot_h^n - \rhot_h^{n-1}, \rhot_h^n\Big) = \norm{\rhot_h^n}^2 - \norm{\rhot_h^{n-1}}^2 +\norm{\rhot_h^n -\rhot_h^{n-1}}^2,
 \eeqs 
 we obtain 
\begin{multline*}
\norm{\rhot_h^n}^2 - \norm{\rhot_h^{n-1}}^2 +\norm{\rhot_h^n -\rhot_h^{n-1}}^2 + 2\Delta t \norm{K^{\frac{1}{2}}(|\nabla \rho_h^n|)\nabla \rho_h^n}^2\\
 = 2\Delta t (\ft^n, \rhot_h^n )+2\Delta t \big(K(|\nabla \rho_h^n|)\nabla \rho_h^n, \nabla \pi\phi^n\big).  
\end{multline*}
Using Young's inequality and \eqref{term1} then  
\begin{align*}
%&(\rhot_h^{n-1}, \rhot_h^n )\le \frac 1 2 \left[    \norm{\rhot_h^{n-1}}^2+\norm{\rhot_h^n}^2  \right],\\
2\Delta t \big(K(|\nabla \rho_h^n|)\nabla \rho_h^n, \nabla \pi\phi^n\big) \le \frac{\Delta t}2  \norm{K^{\frac{1}{2}}(|\nabla \rho_h^n|)\nabla \rho_h^n}^2 +C\Delta t \norm{\nabla \pi\phi^n}^2.
\end{align*}
According to the inequality \eqref{term2}, 
\begin{align*}
2\Delta t (\ft^n, \rhot_h^n )&\le C \Delta t\norm{\ft^n}+ C\Delta t\norm{\ft^n}^{\lambda} +  \frac {\Delta t} 2 \norm{K^{\frac 12}(|\nabla\rho_h^n|) \nabla\rho_h^n}^2 +C \Delta t \norm{\nabla\pi\phi^n}^{2}. 
\end{align*}
Hence,
\beqs
\norm{ \rhot_h^n}^2-\norm{ \rhot_h^{n-1}}^2 +   \norm{\rhot_h^n -\rhot_h^{n-1}}^2 +\Delta t \norm{K^{\frac{1}{2}}(|\nabla \rho_h^n|)\nabla \rho_h^n}^2\le   C\Delta t\left(\norm{\ft^n}+ \norm{\ft^n}^{\lambda}   + \norm{\nabla\pi\phi^n}^2\right).
\eeqs
According to \eqref{trhoEst}, 
\beq\label{preEmb}
c_3 2^{-a}\left(2^{1-\beta}\norm{\nabla \rhot_h^n}_{0,\beta}^\beta  -\norm{\nabla \pi\phi^n}_{0,\beta}^\beta- 1\right)\le \norm{K^{\frac{1}{2}}(|\nabla \rho_h^n|)\nabla \rho_h^n}^2.
\eeq
From the two above inequalities,  we obtain 
\beq\label{rhotn}
\begin{split}
\norm{ \rhot_h^n}^2-\norm{ \rhot_h^{n-1}}^2& + \norm{\rhot_h^n -\rhot_h^{n-1}}^2+ \frac{c_3}{2}\Delta t \norm{\nabla \rhot_h^n}_{0,\beta}^\beta\\
& \le   C\Delta t \left(\norm{\ft^n}+ \norm{\ft^n}^{\lambda} +\norm{\nabla\pi\phi^n}^2+\norm{\nabla\pi\phi^n}_{0,\beta}^\beta+1\right)\\
& \le   C\Delta t \left(1+\norm{\ft^n}^{\lambda} +\norm{\nabla\pi\phi^n}^2\right)\\
&\le   C\Delta t\left(1+ \norm{\ft^n}^{\lambda} +\norm{\phi^n}_{1}^2\right).
\end{split}
\eeq
Applying Poincar\'e inequality  
$
\norm{\rhot_h^n} \le C_p \norm{\nabla \rhot_h^n }_{0,\beta},  
$ 
shows that   
\begin{align*}
\norm{ \rhot_h^n}^2-\norm{ \rhot_h^{n-1}}^2&+\norm{\rhot_h^n -\rhot_h^{n-1}}^2  + \frac{c_3\Delta t}{2C_p}\norm{\rhot_h^n}^\beta\le   C\Delta t\left(1+ \norm{\ft^n}^{\lambda} +\norm{\phi^n}_{1}^2\right).
\end{align*}
Applying the discrete Gronwall's version in Lemma \ref{DGronwall},  we find that 
\begin{align*}
\norm{\rhot_h^n}^2
 \le   C\max \left\{ \norm{\rhot_h^0}^2, \big(1+ \norm{\ft^n}^{\lambda} +\norm{\phi^n}_{1}^2\big)^{\frac{2}{\beta}} \right \},
\end{align*}
which implies \eqref{pwBound}. 

Now summing up \eqref{rhotn} with $n$ from $i$ to $m$ and dropping some nonnegative terms, we find that 
\beqs
\begin{split}
&\frac{c_3}{2}\sum_{j=i}^m \Delta t\norm{ \nabla \rhot_h^j}_{0,\beta}^\beta \le \norm{ \rhot_h^{i-1}}^2 + C\Delta t \sum_{j=i}^m\left(1+ \norm{\ft^j}^{\lambda} +\norm{\phi^j}_{1}^2\right)\\
&\qquad\qquad\le C\max \left\{ \norm{\rhot_h^0}^2, \big(1+ \norm{\ft^{i-1}}^{\lambda} +\norm{\phi^{i-1}}_{1}^2\big)^{\frac{2}{\beta}} \right \}+ C\sum_{j=i}^m\Delta t \left(1+ \norm{\ft^j}^{\lambda} +\norm{\phi^j}_{1}^2\right).
\end{split}
\eeqs
The proof is complete.
\end{proof}

%Let
%$\rhot^n(\cdot) = \rhot(\cdot,t_n)$,  be the
%exact solution evaluated at the discrete time levels.  We will also
%denote $\pi \rho^n \in W_h$ to be the projections of the exact solutions at the discrete time levels.  
As in the semidiscrete case, we use $\chi = \rhot -\rhot_h$, $\vartheta=\rhot-\rhot_h$, $\theta_h=\rhot_h-\pi \rhot$ and $\chi^n$, $\vartheta^n$, $\theta_h^n$  be evaluating $\chi$, $\vartheta$, $\theta_h$ at the discrete time levels.
%===========================
We also define  
$$
\partial \rhot^n = \frac {\rhot^{n} -\rhot^{n-1} }{\Delta t}.
$$

%\subsection{Error estimates }
\begin{theorem}\label{Err-rho-ful}
Let $1\le k\le r+1$,  $\rhot$ solve problem \eqref{weakform} and $\rhot_h^n$ solve the fully  discrete finite element approximation \eqref{fullydiscreteform} for each time step $n$, $n\ge 1$.  Suppose that $\rhot_{tt} \in L^2(\R_+, L^2(\Omega)),$ $\rhot\in L^\infty(\R_+,H^k(\Omega))$. Then, there exists a positive constant $C(\rho)$ independent of $h$ and $\Delta t$  such that if  the $\Delta t$ is sufficiently small then 
\beq\label{fulerrl2}
\norm{\rhot^n-\rhot_h^n}^2 \le C( h^{k-1}+ \Delta t ).
\eeq  
\end{theorem}
%=================================
\begin{proof}
Evaluating equation \eqref{weakform} at $t=t_n$ gives
\beq\label{fuldis1}
\intd  {\rhot_t^{n}, w } +  \intd{K(|\nabla \rho^n|)\nabla \rho^n, \nabla w} =(\ft^n, w ), \quad \forall w\in W. 
\eeq  
Subtracting \eqref{fullydiscreteform} from \eqref{fuldis1}, we obtain 
\beq\label{ErrEq}
( \partial \rho_h^n-\partial \pi\rhot^n , w_h) +  \left(K(|\nabla \rho_h^n|)\nabla \rho_h^n- K(|\nabla \rho^n|)\nabla \rho^n, \nabla w_h\right)=( \pi\rhot_t^n- \partial \pi\rhot^n,w_h ). 
\eeq
We rewrite \eqref{ErrEq} as the form
\beq\label{vvq}
(\partial \theta_h^n , w_h) +  \left(K(|\nabla \rho_h^n|)\nabla \rho_h^n- K(|\nabla \rho^n|)\nabla \rho^n, \nabla w_h\right)=(   \pi\rhot_t^n-\partial \pi\rhot^n,w_h ).  
\eeq  
Selecting $w_h =\theta_h^n$ in \eqref{vvq} gives  
\beq\label{eereq}
\begin{split}
&(\partial \theta_h^n , \theta_h^n ) +  \left(K(|\nabla \rho^n|)\nabla \rho^n- K(|\nabla \rho_h^n|)\nabla \rho_h^n, \nabla \rho ^n- \nabla\rho_h^n \right)\\
&\qquad\qquad=\left(K(|\nabla \rho^n|)\nabla \rho^n- K(|\nabla \rho_h^n|)\nabla \rho_h^n, \nabla \vartheta^n +\nabla \varphi^n \right)+( \pi\rhot_t^n -  \partial \pi\rhot^n,\theta_h^n  ).  
\end{split} 
\eeq
We will evaluate \eqref{eereq}  term by term. 

For the first term, we use the identity 
\beq\label{rhs1}
\begin{split}
(\partial \theta_h^n , \theta_h^n )=\left(\partial \theta_h^n, \frac{\theta_h^n +\theta_h^{n-1}}{2}+ \frac {\Delta t} 2 \partial \theta_h^n\right)= \frac 1{2\Delta t}\left(\norm{\theta_h^n}^2  -\norm{\theta_h^{n-1}}^2\right)+\frac{\Delta t}{2}\norm{\partial \theta_h^n }^2 . 
\end{split}
\eeq

For the second term, the monotonicity of $K(\cdot)$ in \eqref{Mono} and \eqref{InvIneq0} yield 
 \beq\label{rhs2} 
 \begin{split}
 (K(|\nabla \rho^n|)\nabla \rho^n -K(|\nabla \rho_h^n|)\nabla \rho_h^n, \nabla\rho^n -\nabla\rho_h^n)  &\ge c_4\omega^n \norm{\nabla (\rho^n -\rho_h^n)}_{0,\beta}^2\\
 & \ge \frac{c_4}{4C_p^2}\omega^n\norm{\theta_h^n}^2 - \frac{c_4}{2C_p^2}\omega^n\norm{\vartheta^n}^2 - \omega^n\norm{\nabla\varphi^n}_{0,\beta}^2\\
 & \ge \frac{c_4}{4C_p^2}\omega^n\norm{\theta_h^n}^2 - \frac{c_4}{2C_p^2}\norm{\vartheta^n}^2 - \norm{\nabla\varphi^n}_{0,\beta}^2. 
 \end{split}
 \eeq
 where 
 \beqs
 \omega^n= \omega(t_n) = \left(1+\max\left\{\norm{\nabla \rho_h^n}_{0,\beta}, \norm{\nabla \rho^n}_{0,\beta} \right \}\right)^{-a} .
  \eeqs
  
 For third term, using \eqref{r20}, we find that    
\beq\label{lhs1}
\begin{split}
\left(K(|\nabla \rho^n|)\nabla \rho^n- K(|\nabla \rho_h^n|)\nabla \rho_h^n, \nabla \vartheta^n +\nabla \varphi^n \right)\le C \Lambda(t_n) \left( \norm{\nabla \vartheta^n}_{0,\beta}+\norm{ \nabla \varphi^n }_{0,\beta}\right).
\end{split}
\eeq

  For the last term, it follows from using $L^2$-projection and Taylor expand that  
  \beq\label{lhs2}
  \begin{split}
 ( \pi\rho_t^n -  \partial \pi\rho^n,\theta_h^n  )&=( \rho_t^n -  \partial\rho^n,\theta_h^n  ) =\left(\frac 1{\Delta t} \int_{t_{n-1}}^{t_n} \rho_{tt}(\tau) (\tau-t_{n-1})   d\tau, \theta_h^n \right)\\
 &\le \frac 1{\Delta t} \norm{\int_{t_{n-1}}^{t_n} \rho_{tt}(\tau) (\tau-t_{n-1})   d\tau }\norm{\theta_h^n}\\
 &\le \frac {C}{\Delta t }\left(\int_{t_{n-1}}^{t_n} \norm{\rho_{tt}(\tau)}^2 d\tau\right)^{\frac1 2}\left(\int_{t_{n-1}}^{t_n} (\tau-t_{n-1})^2   d\tau\right)^{\frac 12} \norm{\theta_h^n}\\
 &\le C\varep^{-1} \Delta t \int_{t_{n-1}}^{t_n} \norm{\rho_{tt}(\tau)}^2 d\tau +\varep\norm{\theta_h^n}^2 .
 \end{split}
 \eeq
%===========================
  Combining \eqref{rhs1}, \eqref{rhs2},\eqref{lhs1} and \eqref{lhs2}, we obtain     
\begin{multline*}
 \frac 1{2\Delta t}\left(\norm{\theta_h^n}^2  -\norm{\theta_h^{n-1}}^2\right)+  \frac{c_4\omega^n}{4C_p^2} \norm{\theta_h^n}^2\le C \Lambda(t_n) \left( \norm{\nabla \vartheta^n}_{0,\beta}+\norm{ \nabla \varphi^n }_{0,\beta}\right)\\
\quad+ C\varep^{-1} \Delta t \int_{t_{n-1}}^{t_n} \norm{\rho_{tt}(\tau)}^2 d\tau +\varep\norm{\theta_h^n}^2+  C\norm{\vartheta^n}^2 + C\norm{\nabla\varphi^n}_{0,\beta}^2.
\end{multline*}
 Choosing $\varep=\frac{c_4\omega^{n}}{8C_p^2}$ in previous inequality, we find that  
  \beqs
\begin{split}
 &\frac 1{2\Delta t}\left(\norm{\theta_h^n}^2  -\norm{\theta_h^{n-1}}^2\right)+  \frac{c_4\omega^n}{8C_p^2} \norm{\theta_h^n}^2\le C \Lambda(t_n) \left( \norm{\nabla \vartheta^n}_{0,\beta}+\norm{ \nabla \varphi^n }_{0,\beta}\right)\\
 &\quad+ C(\omega^n)^{-1} \Delta t \int_{t_{n-1}}^{t_n} \norm{\rho_{tt}(\tau)}^2 d\tau + C\norm{\vartheta^n}^2 + \norm{\nabla\varphi^n}_{0,\beta}^2\\
 &\le C\Lambda(t_n) \left( \norm{\nabla \vartheta^n}_{0,\beta}+\norm{ \nabla \varphi^n }_{0,\beta}+ \norm{\vartheta^n}^2 + \norm{\nabla\varphi^n}_{0,\beta}^2+ \Delta t \int_{t_{n-1}}^{t_n} \norm{\rho_{tt}(\tau)}^2 d\tau\right).
\end{split}
\eeqs
According to discrete Gronwall's inequality in Lemma~\ref{DGronwall},  
\beq\label{comb3}
\begin{aligned}
\norm{\theta_h^n}^2  &\le C(\omega^n)^{-1}\Lambda(t_n) \left( \norm{\nabla \vartheta^n}_{0,\beta}+\norm{ \nabla \varphi^n }_{0,\beta}+ \norm{\vartheta^n}^2 + \norm{\nabla\varphi^n}_{0,\beta}^2+ \Delta t \int_{t_{n-1}}^{t_n} \norm{\rho_{tt}(\tau)}^2 d\tau\right)\\
 &\le C\Lambda(t_n)^2 \left( \norm{\nabla \vartheta^n}_{0,\beta}+\norm{ \nabla \varphi^n }+ \norm{\vartheta^n}^2 + \norm{\nabla\varphi^n}_{0,\beta}^2+ \Delta t \int_{t_{n-1}}^{t_n} \norm{\rho_{tt}(\tau)}^2 d\tau\right)\\
 &\le C\Ff^2(t_n) \left( \norm{\nabla \vartheta^n}_{0,\beta}+\norm{ \nabla \varphi^n }+ \norm{\vartheta^n}^2 + \norm{\nabla\varphi^n}_{0,\beta}^2+ \Delta t \int_{t_{n-1}}^{t_n} \norm{\rho_{tt}(\tau)}^2 d\tau\right).
\end{aligned}
\eeq
The inequality \eqref{fulerrl2} follows from combining \eqref{comb3} and the inequality
\beq\label{comb5}
 \norm{\rhot^n-\rhot_h^n}^2\le 2\norm{\theta_h^n}^2 +2\norm{\vartheta^n}^2. 
 \eeq
The proof is complete.  
\end{proof}

\begin{theorem} Under assumption of Theorem~\ref{Err-rho-ful}. There exists a positive constant $C(\rho,\phi)$ independent of $h$ and $\Delta t$  such that if  the $\Delta t$ is sufficiently small then 
 \beq\label{rateConv2}
 \norm{\nabla \rho^n_h - \nabla \rho^n}_{0,\beta}^2 \le C (h^{k-1}+\Delta t). 
 \eeq
 
\end{theorem}

\begin{proof}
We rewrite \eqref{ErrEq} with $w_h=\theta_h$ as 
\begin{multline*}
  \left(K(|\nabla \rho^n|)\nabla \rho^n- K(|\nabla \rho_h^n|)\nabla \rho_h^n, \nabla \rho ^n- \nabla\rho_h^n \right)\\
=\left(K(|\nabla \rho^n|)\nabla \rho^n- K(|\nabla \rho_h^n|)\nabla \rho_h^n, \nabla \vartheta^n +\nabla \varphi^n \right)+( \rho_t^n -  \partial \rho^n,\theta_h^n  ).
\end{multline*} 
Due to \eqref{rhs2}, \eqref{lhs1} and \eqref{lhs2}, we have  
\beqs
\begin{split}
c_4\omega^n\norm{\nabla \rho^n_h - \nabla \rho^n}_{0,\beta}^2 &\le C \Lambda(t_n) \left( \norm{\nabla \vartheta^n}_{0,\beta}+\norm{ \nabla \varphi^n }_{0,\beta}\right)\\
&+C\varep^{-1} \Delta t \int_{t_{n-1}}^{t_n} \norm{\rho_{tt}(\tau)}^2 d\tau +\varep\norm{\theta_h^n}^2.
\end{split}
\eeqs
Using \eqref{rhs2}, we find that
\beqs
\varep\norm{\theta_h^n}^2 \le 4\varep C_p^2 \norm{\nabla \rho^n_h - \nabla \rho^n}_{0,\beta}^2 +   C\varep(\omega^n)^{-1}\left(\norm{\vartheta^n}^2 + \norm{\nabla\varphi^n}_{0,\beta}^2\right).
\eeqs
Hence, 
\beq
\begin{split}
c_4\omega^n\norm{\nabla \rho^n_h - \nabla \rho^n}_{0,\beta}^2 \le C \Lambda(t_n) \left( \norm{\nabla \vartheta^n}_{0,\beta}+\norm{ \nabla \varphi^n }_{0,\beta}\right)+ C\varep^{-1} \Delta t \int_{t_{n-1}}^{t_n} \norm{\rho_{tt}(\tau)}^2 d\tau\\ +4\varep C_p^2 \norm{\nabla \rho^n_h - \nabla \rho^n}_{0,\beta}^2 +   C\varep(\omega^n)^{-1}\left(\norm{\vartheta^n}^2 + \norm{\nabla\varphi^n}_{0,\beta}^2\right). 
\end{split}
\eeq
Selecting $\varep = \frac{c_4\omega^n}{8C_p^2}$ in previous inequality gives    
\begin{multline*}
\norm{\nabla \rho^n_h - \nabla \rho^n}_{0,\beta}^2 \le C \Lambda(t_n)(\omega^n)^{-1} \left( \norm{\nabla \vartheta^n}_{0,\beta}+\norm{ \nabla \varphi^n }_{0,\beta}\right)\\
+ C(\omega^n)^{-2} \Delta t \int_{t_{n-1}}^{t_n} \norm{\rho_{tt}(\tau)}^2 d\tau  +   C(\omega^n)^{-1}\left(\norm{\vartheta^n}^2 + \norm{\nabla\varphi^n}_{0,\beta}^2\right). 
\end{multline*}
Then using \eqref{Odef0}, we have $(\omega^n)^{-1}\le \Lambda(t_n)$. Note that $\Lambda(t_n)>1$, thus  
\beqs
\begin{split}
\norm{\nabla \rho^n_h - \nabla \rho^n}_{0,\beta}^2 
\le C \Lambda(t_n)^2 \left( \norm{\nabla \vartheta^n}_{0,\beta}+\norm{ \nabla \varphi^n }+  \Delta t \int_{t_{n-1}}^{t_n} \norm{\rho_{tt}(\tau)}^2 d\tau+\norm{\vartheta^n}^2 + \norm{\nabla\varphi^n}_{0,\beta}^2\right)\\
\le C \Ff^2(t_n) \left( \norm{\nabla \vartheta^n}_{0,\beta}+\norm{ \nabla \varphi^n }+  \Delta t \int_{t_{n-1}}^{t_n} \norm{\rho_{tt}(\tau)}^2 d\tau+\norm{\vartheta^n}^2 + \norm{\nabla\varphi^n}_{0,\beta}^2\right). 
\end{split}
\eeqs
This proves \eqref{rateConv2}. The proof is complete. 
\end{proof}
%=======================================

\section{Numerical results} \label{Num-result}

In this section, we give simple numerical experiments using Galerkin finite element method in the two dimensional region to illustrate the convergent theory. We test the convergence of our method with the Forchheimer two-term law $g(s)=1+s$. Equation \eqref{Kdef} $sg(s)=\xi,$ $s\ge0$ gives
 $ s= \frac {-1 +\sqrt{1+4\xi}}{2}$ and hence 
$$
K(\xi) =\frac1{g(s(\xi)) } =\frac {2}{1+\sqrt{1+4\xi}}.
$$ 

{\bf Example 1.} The analytical solution is as follows  
\beqs
\rho(x,t)=e^{-2t}x_1(1-x_1)x_2(1-x_2), \quad \forall (x,t)\in [0,1]^2\times[0,1].  
\eeqs
The forcing term $f$ is determined accordingly to the analytical solution by equation $p_t - \nabla \cdot (K(|\nabla \rho|)\nabla\rho ) = f$. Explicitly,  
\beqs
\begin{aligned}
& f(x,t)=-2e^{-2t}x_1(1-x_1)x_2(1-x_2)+\frac {4e^{-2t}\big[x_2(1-x_2) +x_1(1-x_1)\big]}{1+\sqrt{1+4e^{-2t}w(x)  }}\\
&+\frac{2e^{-4t}x_2(1-x_2)(1-2x_1)\big[2x_1(1-x_1)^2(1-2x_2)^2-2x_1^2(1-x_1)(1-2x_2)^2-4x_2^2(1-x_2)^2(1-2x_1) \big] }{ w(x) \sqrt{1+4e^{-2t}w(x)}\left(1+\sqrt{1+4e^{-2t}w(x)}\right)^2  }\\
&+\frac{2e^{-4t}x_1(1-x_1)(1-2x_2)\big[2x_2(1-x_2)^2(1-2x_1)^2 -2x_2^2(1-x_2)(1-2x_1)^2 -4x_1^2(1-x_1)^2(1-2x_2) \big]}{ w(x) \sqrt{1+4e^{-2t}w(x)}\left(1+\sqrt{1+4e^{-2t}w(x)}\right)^2  },
\end{aligned}
\eeqs
 where  $w(x) = \sqrt{(x_2(1-x_2)(1-2x_1))^2 +(x_1(1-x_1)(1-2x_2))^2}$. 
 
The initial data $\rho^0(x) =x_1(1-x_1)x_2(1-x_2)$ and the boundary data $\psi(x,t)=0.$\\
We use the Lagrange element of order $r=1$ on the unit square in two dimensions. Our problem is solved at each time level starting at $t=0$ until the given final time $T$.  At time $T$, we measured the error in $L^2$-norm for density and $L^\beta$-norm for the gradient density. In this example $\beta = 2 - a = 2- \frac{{\mathrm deg} (g)}{{\mathrm deg} ( g) + 1}=\frac 32$. The numerical results are listed in Table I.

\vspace{0.1cm}  
\begin{center}
\begin{tabular}{l| c| c| c| c}
\hline
\multicolumn{5}{ c }{At time $T=1$}\\
\hline
N    & \qquad $\norm{\rho-\rho_h}$ \qquad  &  \quad Rates  \quad&  \quad$\norm{\nabla(\rho-\rho_h)}_{0,\beta}$ \quad & \quad  Rates \quad  \\
\hline
4	&$1.668E-02$      	&-   	        		& $7.081E-02$		&-\\
8	&$1.049E-02$	   	&$0.669$  	& $4.654E-02$  	&$0.605$\\
16	&$6.004E-03$		&$0.805$    	& $2.741E-02$ 		&$0.764$\\
32	&$3.272E-03$		&$0.876$   	& $1.530E-02$ 		&$0.841$\\
64	&$1.723E-03$		&$0.926$		& $8.277E-03$		&$0.887$\\
128  &$8.889E-04$       	&$0.954$		& $4.411E-03$ 		&$0.908$\\
256  &$4.531E-04$       	&$0.972$		& $2.336E-03$		&$0.917$\\
\hline
\multicolumn{5}{ c }{At time $T=10$}\\
\hline
N    & \qquad $\norm{\rho-\rho_h}$ \qquad  &  \quad Rates  \quad&  \quad$\norm{\nabla(\rho-\rho_h)}_{0,\beta}$ \quad & \quad  Rates \quad  \\
\hline
4	&$1.851E-02 $        	&-          		& $7.818E-02 $		&-\\
8	&$1.194E-02 $	   	&$0.633 $  	& $ 5.247E-02$  	&$ 0.575$\\
16	&$6.892E-03 $	   	&$ 0.792$    	& $3.114E-02 $ 	&$0.753 $\\
32	&$3.782E-03 $		&$0.866 $  	& $1.752E-02 $ 	&$ 0.830$\\
64	&$2.001E-03 $	   	&$0.918 $		& $9.537E-03 $		&$ 0.877$\\
128	&$1.035E-03 $		&$0.951 $		& $5.104E-03 $ 	&$0.902 $\\
256	&$ 5.280E-04$  	&$0.971$		& $ 2.713E-03$		&$0.912$\\
\hline
\end{tabular}

\vspace{0.1cm}
Table I. {\it Convergence study for generalized Forchheimer flows using Galerkin FEM with zero  boundary data in 2D.}
\end{center}

{\bf Example 2.} The analytical solution is $\rho(x,t)=e^{1-t}(x_1^2+x_2^2)$ for all  $(x,t)\in [0,1]^2\times [0,1]$. The forcing term $f$, initial condition and boundary condition are determined accordingly to the analytical solution as follows 
\begin{align*}
&f(x,t)= -e^{1-t}z(x)+\frac{16e^{2-2t}z(x) }{\sqrt{z(x)} \sqrt{1+8e^{1-t}\sqrt{z(x)}}\left(1+\sqrt{1+8e^{1-t}\sqrt{z(x)}}\right)^2 }-\frac{8e^{1-t}}{1+\sqrt{1+8e^{1-t}\sqrt{z(x)}}},\\
&\hspace{3cm}\rho^0(x) =e\cdot z(x), \quad \psi(x,t)=e^{1-t}
\begin{cases} 
x_2^2 &\text { on } x_1=0,\\
1+x_2^2 & \text { on } x_1=1,\\
1+x_1^2 & \text { on } x_2=1,\\
x_1^2 & \text { on } x_2=0,
 \end{cases}
\end{align*}
where $z(x)=x_1^2+x_2^2$.
The numerical results are listed in Table II 

 \vspace{0.1cm}  
\begin{center}
\begin{tabular}{l| c| c| c| c}
\hline
\multicolumn{5}{ c }{At time $T=1$}\\
\hline
N    & \qquad $\norm{\rho-\rho_h}$ \qquad  &  \quad Rates  \quad&  \quad$\norm{\nabla(\rho-\rho_h)}_{0,\beta}$ \quad & \quad  Rates \quad  \\
\hline
4	&$7.785E-02 $        	&-          		& $3.767E-01 $		&-\\
8	&$5.700E-02 $	   	&$0.450 $  	& $2.972E-01$  	&$0.342$\\
16	&$ 2.870E-02$	   	&$0.990 $    	& $1.918E-01$ 		&$0.631 $\\
32	&$ 2.296E-02$		&$ 0.322$  	& $1.038E-01 $ 	&$0.887 $\\
64	&$1.434E-02 $	   	&$0.679 $		& $ 6.029E-02$		&$0.783 $\\
128	&$5.459E-03$		&$1.393$		& $ 3.641E-02$		&$ 0.728$\\
256	&$2.401E-03$  		&$1.185$		& $2.128E-02 $		&$0.775 $\\
\hline
\multicolumn{5}{ c }{At time $T=10$}\\
\hline
N    & \qquad $\norm{\rho-\rho_h}$ \qquad  &  \quad Rates  \quad&  \quad$\norm{\nabla(\rho-\rho_h)}_{0,\beta}$ \quad & \quad  Rates \quad  \\
\hline
4	&$9.335E-06$        	&-          		& $5.502E-05 $	&-\\
8	&$6.829E-06 $	   	&$ 0.451$  	& $4.050E-05 $  	&$0.442 $\\
16	&$4.983E-06 $	   	&$0.455 $    	& $2.345E-05 $ 	&$0.788 $\\
32	&$3.014E-06 $		&$0.726 $  	& $1.381E-05 $ 	&$ 0.764$\\
64	&$ 1.351E-06 $	   	&$1.158 $		& $8.231E-06 $		&$0.747 $\\
128	&$6.447E-07 $		&$1.067 $		& $4.759E-06 $ 	&$0.790 $\\
256	&$3.935E-07 $  	&$0.712 $		& $2.666E-06 $		&$ 0.836$\\
\hline
\end{tabular}

\vspace{0.3cm}
Table II. {\it Convergence study for generalized Forchheimer flows using Galerkin FEM with nonzero Dirichlet boundary data in 2D.}
\end{center}

%====================================================

\myclearpage
\myclearpage
\appendix

%%%%%%%%%%%%%%%%%%%%%%%%%%%%%%%%%%%%%%%%%%%%%%%%%%%%%%%%%%%%%%%%%%%%%
% Bibliography using BibTeX
%%%%%%%%%%%%%%%%%%%%%%%%%%%%%%%%%%%%%%%%%%%%%%%%%%%%%%%%%%%%%%%%%%%%%
%%%%%%%%%%%%%%%%%%%%%%%%%%%%%%%%%%%%%%%%%%%%%%%%%%%%%%%%%%%%%%%%%%%%%
\bibliographystyle{siam}
%\bibliography{Fpapers,paperbase4}{}

\begin{thebibliography}{10}

\bibitem{ATWZ96}
{\sc T.~Arbogast, M.~F. Wheeler, and N.-Y. Zhang}, {\em A nonlinear mixed
  finite element method for a degenerate parabolic equation arising in flow in
  porous media}, SIAM J. Numer. Anal., 33 (1996), pp.~1669--1687.

\bibitem{ABHI1}
{\sc E.~Aulisa, L.~Bloshanskaya, L.~Hoang, and A.~Ibragimov}, {\em {Analysis of
  generalized {F}orchheimer flows of compressible fluids in porous media}}, J.
  Math. Phys., 50 (2009), pp.~103102, 44.

\bibitem{BearBook}
{\sc J.~Bear}, {\em Dynamics of Fluids in Porous Media}, Dover, New York, 1972.

\bibitem{BPS02}
{\sc J.~H. Bramble, J.~E. Pasciak, and O.~Steinbach}, {\em On the stability of
  the l2 projection in h1(Ω)}, Mathematics of Computation, 71 (2002), pp.~pp.
  147--156.

\bibitem{BS08}
{\sc S.~C. Brenner and L.~R. Scott}, {\em The mathematical theory of finite
  element methods}, vol.~15 of Texts in Applied Mathematics, Springer, New
  York, third~ed., 2008.

\bibitem{BF91}
{\sc F.~Brezzi and M.~Fortin}, {\em Mixed and hybrid finite element methods},
  vol.~15 of Springer Series in Computational Mathematics, Springer-Verlag, New
  York, 1991.

\bibitem{Ciarlet78}
{\sc P.~G. Ciarlet}, {\em The finite element method for elliptic problems},
  North-Holland Publishing Co., Amsterdam, 1978.
\newblock Studies in Mathematics and its Applications, Vol. 4.

\bibitem{MR2566733}
{\sc E.~DiBenedetto}, {\em Partial differential equations}, Cornerstones,
  Birkh{\"a}user Boston Inc., Boston, MA, second~ed., 2010.

\bibitem{Doug1993}
{\sc J.~J. Douglas, P.~J. Paes-Leme, and T.~Giorgi}, {\em Generalized
  {F}orchheimer flow in porous media}, in Boundary value problems for partial
  differential equations and applications, vol.~29 of RMA Res. Notes Appl.
  Math., Masson, Paris, 1993, pp.~99--111.

\bibitem{F06}
{\sc K.~B. Fadimba}, {\em Error analysis for a {G}alerkin finite element method
  applied to a coupled nonlinear degenerate system of advection-diffusion
  equations}, Comput. Methods Appl. Math., 6 (2006), pp.~3--30 (electronic).

\bibitem{F07}
\leavevmode\vrule height 2pt depth -1.6pt width 23pt, {\em On existence and
  uniqueness for a coupled system modeling immiscible flow through a porous
  medium}, J. Math. Anal. Appl., 328 (2007), pp.~1034--1056.

\bibitem{FS95}
{\sc K.~B. Fadimba and R.~C. Sharpley}, {\em A priori estimates and
  regularization for a class of porous medium equations}, Nonlinear World, 2
  (1995), pp.~13--41.

\bibitem{FS04}
\leavevmode\vrule height 2pt depth -1.6pt width 23pt, {\em Galerkin finite
  element method for a class of porous medium equations}, Nonlinear Anal. Real
  World Appl., 5 (2004), pp.~355--387.

\bibitem{ForchheimerBook}
{\sc P.~Forchheimer}, {\em Wasserbewegung durch Boden Zeit}, vol.~45, Ver.
  Deut. Ing., 1901.

\bibitem{HI1}
{\sc L.~Hoang and A.~Ibragimov}, {\em Structural stability of generalized
  {F}orchheimer equations for compressible fluids in porous media},
  Nonlinearity, 24 (2011), pp.~1--41.

\bibitem{HI2}
{\sc L.~Hoang and A.~Ibragimov}, {\em {Qualitative Study of Generalized
  {F}orchheimer Flows with the Flux Boundary Condition}}, Adv. Diff. Eq., 17
  (2012), pp.~511--556.

\bibitem{HIKS1}
{\sc L.~Hoang, A.~Ibragimov, T.~Kieu, and Z.~Sobol}, {\em Stability of
  solutions to generalized {F}orchheimer equations of any degree}, J. Math.
  Sci., 210 (2015), pp.~476--544.

\bibitem{HK1}
{\sc L.~Hoang and T.~Kieu}, {\em {Interior estimates for generalized
  {F}orchheimer flows of slightly compressible fluids}},  (2014).
\newblock submitted, preprint http://arxiv.org/abs/1404.6517.

\bibitem{HK2}
\leavevmode\vrule height 2pt depth -1.6pt width 23pt, {\em Global estimates for
  generalized {F}orchheimer flows of slightly compressible fluids}, Journal
  d'Analyse Mathematique,  (2015).
\newblock accepted.

\bibitem{HIK1}
{\sc L.~T. Hoang, A.~Ibragimov, and T.~T. Kieu}, {\em One-dimensional two-phase
  generalized {F}orchheimer flows of incompressible fluids}, J. Math. Anal.
  Appl., 401 (2013), pp.~921--938.

\bibitem{HIK2}
\leavevmode\vrule height 2pt depth -1.6pt width 23pt, {\em A family of steady
  two-phase generalized {F}orchheimer flows and their linear stability
  analysis}, J. Math. Phys., 55 (2014), p.~123101.

\bibitem{HKP1}
{\sc L.~T. Hoang, T.~T. Kieu, and T.~V. Phan}, {\em {Properties of generalized
  {F}orchheimer flows in porous media}}, J. Math. Sci., 202 (2014),
  pp.~259--332.

\bibitem{JK95}
{\sc D.~Jerison and C.~E. Kenig}, {\em The inhomogeneous {D}irichlet problem in
  {L}ipschitz domains}, J. Funct. Anal., 130 (1995), pp.~161--219.

\bibitem{JT81}
{\sc C.~Johnson and V.~Thom{\'e}e}, {\em Error estimates for some mixed finite
  element methods for parabolic type problems}, RAIRO Anal. Num\'er., 15
  (1981), pp.~41--78.

\bibitem{NJ2000}
{\sc N.~Ju}, {\em Numerical analysis of parabolic {$p$}-{L}aplacian:
  approximation of trajectories}, SIAM J. Numer. Anal., 37 (2000),
  pp.~1861--1884 (electronic).

\bibitem{K1}
{\sc T.~Kieu}, {\em Analysis of expanded mixed finite element methods for the
  generalized {F}orchheimer flows of slightly compressible fluids}, Numer.
  Methods Partial Differential Equations,  (2015).
\newblock accepted.

\bibitem{LadyParaBook68}
{\sc O.~A. Lady\v{z}enskaja, V.~A. Solonnikov, and N.~N. Ural{\cprime}ceva},
  {\em Linear and quasilinear equations of parabolic type}, Translated from the
  Russian by S. Smith. Translations of Mathematical Monographs, Vol. 23,
  American Mathematical Society, Providence, R.I., 1968.

\bibitem{MR0259693}
{\sc J.-L. Lions}, {\em Quelques m{\'e}thodes de r{\'e}solution des
  probl{\`e}mes aux limites non lin{\'e}aires}, Dunod, 1969.

\bibitem{MC70}
{\sc C.~Miranda}, {\em Partial differential equations of elliptic type},
  Ergebnisse der Mathematik und ihrer Grenzgebiete, Band 2, Springer-Verlag,
  New York-Berlin, 1970.
\newblock Second revised edition. Translated from the Italian by Zane C.
  Motteler.

\bibitem{Muskatbook}
{\sc M.~Muskat}, {\em The flow of homogeneous fluids through porous media},
  McGraw-Hill Book Company, inc., 1937.

\bibitem{EJP05}
{\sc E.-J. Park}, {\em Mixed finite element methods for generalized
  {F}orchheimer flow in porous media}, Numer. Methods Partial Differential
  Equations, 21 (2005), pp.~213--228.

\bibitem{s97}
{\sc R.~E. Showalter}, {\em Monotone operators in {B}anach space and nonlinear
  partial differential equations}, vol.~49 of Mathematical Surveys and
  Monographs, American Mathematical Society, Providence, RI, 1997.

\bibitem{TV06}
{\sc V.~Thom{\'e}e}, {\em Galerkin finite element methods for parabolic
  problems}, vol.~25 of Springer Series in Computational Mathematics,
  Springer-Verlag, Berlin, second~ed., 2006.

\bibitem{TW06}
{\sc F.~Tone and D.~Wirosoetisno}, {\em On the long-time stability of the
  implicit {E}uler scheme for the two-dimensional {N}avier-{S}tokes equations},
  SIAM J. Numer. Anal., 44 (2006), pp.~29--40.

\bibitem{Ward64}
{\sc J.~C. Ward}, {\em Turbulent flow in porous media.}, Journal of the
  Hydraulics Division, Proc. Am. Soc. Civ. Eng., 90(HY5) (1964), pp.~1--12.

\bibitem{WCD00}
{\sc C.~S. Woodward and C.~N. Dawson}, {\em Analysis of expanded mixed finite
  element methods for a nonlinear parabolic equation modeling flow into
  variably saturated porous media}, SIAM J. Numer. Anal., 37 (2000),
  pp.~701--724 (electronic).

\bibitem{z90}
{\sc E.~Zeidler}, {\em Nonlinear functional analysis and its applications.
  {II}/{B}}, Springer-Verlag, New York, 1990.
\newblock Nonlinear monotone operators, Translated from the German by the
  author and Leo F. Boron.

\end{thebibliography}
\def\cprime{$'$} \def\cprime{$'$} \def\cprime{$'$}

\end{document}